\documentclass[12pt,reqno]{amsart}
\usepackage{fullpage}

\newtheorem{theorem}{Theorem}[section]
\newtheorem{lemma}[theorem]{Lemma}

\newtheorem{cor}[theorem]{Corollary}

\usepackage{graphicx}
\usepackage{color}
\usepackage[dvipsnames]{xcolor}
\usepackage{subfigure}
\usepackage{amssymb}
\usepackage{amsmath,mathrsfs}
\usepackage{colonequals}
\usepackage[backref=page]{hyperref}
\usepackage[utf8x]{inputenc}
\usepackage{rotating}
\usepackage{multirow}

\theoremstyle{definition}
\newtheorem{definition}[theorem]{Definition}

\newtheorem{example}[theorem]{Example}

\newtheorem{remark}[theorem]{Remark}
\newtheorem{question}[theorem]{Question}

\renewcommand{\subset}{\subseteq}

\renewcommand{\epsilon}{\varepsilon}

\newcommand{\abs}[1]{\left|#1\right|}                   
\newcommand{\absf}[1]{|#1|}                             

\newcommand{\E}{\mathbb{E}}

\renewcommand{\P}{\mathbb{P}}
\newcommand{\R}{\mathbb{R}}

\newcommand{\embolden}[1]{\textbf {#1}}

\begin{document}
\title{Independent Sets of Random Trees\\ and of Sparse Random Graphs}

\author{Steven Heilman}
\address{Department of Mathematics, University of Southern California, Los Angeles, CA 90089-2532}
\email{stevenmheilman@gmail.com}
\date{\today}
\thanks{Supported by NSF Grants DMS 1839406 and CCF 1911216.}

\begin{abstract}
An independent set of size $k$ in a finite undirected graph $G$ is a set of $k$ vertices of the graph, no two of which are connected by an edge.  Let $x_{k}(G)$ be the number of independent sets of size $k$ in the graph $G$ and let $\alpha(G)=\max\{k\geq0\colon x_{k}(G)\neq0\}$.  In 1987, Alavi, Malde, Schwenk and Erd\"{o}s asked if the independent set sequence $x_{0}(G),x_{1}(G),\ldots,x_{\alpha(G)}(G)$ of a tree is unimodal (the sequence goes up and then down).  This problem is still open.  In 2006, Levit and Mandrescu showed that the last third of the independent set sequence of a tree is decreasing.  We show that the first 46.8\% of the independent set sequence of a random tree is increasing with (exponentially) high probability as the number of vertices goes to infinity.  So, the question of Alavi, Malde, Schwenk and Erd\"{o}s is ``four-fifths true'', with high probability.

We also show unimodality of the independent set sequence of Erd\"{o}s-Renyi random graphs, when the expected degree of a single vertex is large (with (exponentially) high probability as the number of vertices in the graph goes to infinity, except for a small region near the mode).  A weaker result is shown for random regular graphs.

The structure of independent sets of size $k$ as $k$ varies is of interest in probability, statistical physics, combinatorics, and computer science.
\end{abstract}
\maketitle

\setcounter{tocdepth}{1}
\tableofcontents
%
%
%

\section{Introduction}\label{secintro}

Let $G$ be a finite, undirected graph with no self-loops and no multiple edges, on $n\geq1$ labelled vertices $V\colonequals\{1,\ldots,n\}$, with edges $E\subset\{\{i,j\}\colon i,j\in V,\,i\neq j\}$.  Let $0\leq k\leq n$.  An \textbf{independent set} of size $k$ in $G$ is a subset of vertices no two of which are connected by an edge.  Let $x_{k}=x_{k}(G)$ denote the number of independent sets in size $k$ in $G$.  (Note that $x_{0}(G)=1$ since we consider the empty set to be a subset of $V=\{1,\ldots,n\}$.)  We refer to the sequence $x_{0}(G),\ldots,x_{n}(G)$ as the \textbf{independent set sequence} of $G$.

The general question considered in this paper is:  What does the sequences $x_{0},x_{1},\ldots,x_{n}$ ``look like'' for random graphs?

Some motivations for this question include:
\begin{itemize}
\item Statistical physics, where an independent set represents molecules in a magnet that do not want to be close to each other;
\item Computer Science and Probability, where many combinatorial optimization problems ($k$-SAT, MAX-Independent Set, etc.) exhibit similar interesting behavior for random instances. See for example the phase transition known as ``shattering'' discussed in e.g. \cite{coja15} and \cite{ding16}.
\item Combinatorics, where one would like to exactly or approximately find the numbers $x_{0},\ldots,x_{n}$ for both deterministic and random graphs.
\end{itemize}

We note that a classic NP-complete problem is:  For any $k\geq1$ and any graph $G$, decide whether or not $x_{k}(G)>0$.  So, it could be hard for a computer to decide whether or not a large graph has a large independent set.  Counting the \textit{number} of independent sets of a given size is then computationally more difficult.  In fact, a remarkable result of \cite{jerrum04} implies a computational equivalence between approximately counting combinatorial quantities (such as independent sets), and constructing a stochastic process whose distribution converges to the uniform distribution on those combinatorial objects.

Intuitively, $x_{0},\ldots,x_{n}$ should resemble the binomial coefficients $\binom{n}{0},\binom{n}{1},\ldots,\binom{n}{n}$.  Note that if $E=\emptyset$, $x_{k}=\binom{n}{k}$ for all $0\leq k\leq n$.  The sequence of binomial coefficients is both unimodal and log-concave, so one might expect the independent set sequence of a random graph to have this same behavior.

\begin{definition}
We say a sequence $a_{0},\ldots,a_{n}$ of real numbers is \textbf{unimodal} if there exists $0\leq j\leq n$ such that
$$a_{0}\leq a_{1}\leq\cdots\leq a_{j}\geq a_{j+1}\geq a_{j+2}\geq\cdots\geq a_{n}.$$
We say a sequence $a_{0},\ldots,a_{n}$ of real numbers is \textbf{log-concave} if
$$a_{k}^{2}\geq a_{k+1}a_{k-1},\qquad\forall\,1\leq k\leq n-1.$$
\end{definition}
\begin{remark}
A positive log-concave sequence of real numbers is unimodal. 
\end{remark}

The following example, demonstrated in \cite{alavi87}, shows that the independent set sequence of a graph might not be unimodal.

\begin{example}\label{ex1}
For any integer $n\geq1$, let $K_{n}$ denote the complete graph on $n$ vertices.  Consider the graph $G= K_{37}+3 K_{4}$ formed by connecting all vertices of a $K_{37}$ to all vertices of three disjoint $K_{4}$'s.  Then
$$x_{1}=37+4\cdot 4=49,\qquad
x_{2}=4\cdot 4\cdot 3=48,\qquad
x_{3}=4^{3}=64.$$
That is, the independent set sequence of $G$ is not unimodal.
\end{example}

In fact, the independent set sequence can increase or decrease in any prescribed way:

\begin{theorem}[{\cite{alavi87}}]\label{althm}
Let $\pi\colon\{1,\ldots,j\}\to\{1,\ldots,j\}$ be any permutation.  Then there exists a graph $G$ whose largest independent set is of size $j$, such that
$$x_{\pi(1)}(G)< x_{\pi(2)}(G)<\cdots<x_{\pi(j)}(G).$$
\end{theorem}

Theorem \ref{althm} is proven by iterating the construction in Example \ref{ex1}.  Perhaps motivated by the odd behavior observed in Theorem \ref{althm}, where many edges in the graph can arbitrarily distort the independent set sequence, it was asked in \cite{alavi87} if trees or forests have a unimodal independent set sequence.

\begin{definition}
A \textbf{tree} on $n$ vertices is a connected graph with no cycles.  A \textbf{forest} is a disjoint union of trees.
\end{definition}

\begin{question}[{\cite{alavi87}}]\label{conj1}
Any tree or forest has a unimodal independent set sequence.
\end{question}
\begin{remark}[{\cite{alavi87}}]
A disjoint union of graphs with unimodal independent set sequence might not have a unimodal independent set sequence (the independent set sequence of a disjoint union is the convolution of the two separate sequences, and it can occur that the convolution of two unimodal sequences is not be unimodal).  So, the tree case of Question \ref{conj1} does not imply the forest case.
\end{remark}

Despite much effort, including \cite{in2,in3,in5,in4,galvin11,galvin12,in1,in6,in9}, Question \ref{conj1} remains open.  The cited works mostly focus on answering Question \ref{conj1} for particular families of trees. One general partial result towards Question \ref{conj1} is the following.

\begin{theorem}[{\cite{levit07}}]\label{levitthm}
Let $T$ be a tree whose largest independent set is of size $j$.  Then the ``last third'' of the independence set sequence is unimodal:
$$x_{\lceil(2j-1)/3\rceil}(T)\geq x_{1+\lceil(2j-1)/3\rceil}(T)\geq\cdots\geq x_{j-1}(T)\geq x_{j}(T).$$
\end{theorem}

That is, Question \ref{conj1} is ``one-third true.''  The proof of Theorem \ref{levitthm} works for any bipartite graph, using that such graphs are perfect with clique number at most $2$.  It was therefore asked in \cite{levit07} whether all bipartite graphs have unimodal independent set sequence, though this was proven false.

\begin{theorem}[{\cite{kahn13}}]
There exists a bipartite graph whose independent set sequence is not unimodal.
\end{theorem}

Due to the apparent difficulty of Question \ref{conj1}, Galvin asked the (potentially) easier question: Is it possible to answer Question \ref{conj1} for \textit{random} trees or \textit{random} forests (with high probability)?   As a ``first approximation'' to random trees, one can also try to prove unimodality of sparse Erd\"{o}s-Renyi random graphs, since these graphs are known to be locally tree-like (see e.g. \cite{grimmett80}).


Here is one such result in this direction. Let $V=\{1,\ldots,2n\}$.  Let $0<p<1$.  Let $E$ be a random subset of $\{\{i,j\}\in V\times V\colon 1\leq i\leq n<j\leq 2n\}$ such that
$$\P(\{i,j\}\in E)=p,\qquad\forall\,1\leq i\leq n<j\leq 2n$$
and such that the events $\{\{i,j\}\in E\}_{1\leq i\leq n<j\leq 2n}$ are independent.  Then $G=(V,E)$ is an \textbf{Erd\"{o}s-Renyi} random bipartite graph on $n$ vertices with parameter $0<p<1$.  This random graph is sometimes denoted as $G=G(n,n,p)$.

\begin{theorem}[{\cite{galvin11,galvin12}}]\label{galvthm}
There exists $c>0$ such that, almost surely, as $n\to\infty$, if $p\geq c\log n/\sqrt{n}$, then $G(n,n,p)$ has unimodal independent set sequence.
\end{theorem}

When the parameter $p$ is around $2/n$, the Erd\"{o}s-Renyi random bipartite graphs closely resemble trees, since in both cases the expected degree of a fixed vertex is around $2$.  (Though the Erd\"{o}s-Renyi random graph will have a constant fraction of its vertices with degree zero, unlike a tree.)  One main difficulty in proving a result such as Theorem  \ref{galvthm} is: when this random graph $G$ is sparse (i.e. it does not have many edges),  $\mathrm{Var}[x_{\alpha n}(G)]$ is exponentially large as $n\to\infty$ (for fixed small $\alpha>0$).  So, ``standard'' concentration of measure results such as Chebyshev's inequality do not directly help.  (On the other hand, if one adjusts the probabilistic model in various ways, then Chebyshev's inequality can sometimes apply; see e.g. \cite{dani11,ding16}.)

In fact, $x_{k}(G)$ does not concentrate around its expected value at all (see e.g. \cite[Corollary 19]{coja15}.)  So, it seems difficult to apply concentration of measure techniques directly to $x_{k}(G)$.  We get around this issue by proving concentration of measure results for the number of vertices not connected to a given independent set, using a conditioning argument of \cite{coja15}.    In this way, the vertex expansion property of the random graph plays a key role in the proof.

\subsection{Independence Polynomials}

A property of sequences that is stronger than log-concavity is the real-rooted property.

\begin{remark}[Newton]\label{newtonrk}
Let $a_{0},\ldots,a_{n}\in\R$.  If the polynomial $t\mapsto\sum_{k=0}^{n}a_{k}t^{k}$ has all real roots, then $a_{0},\ldots,a_{n}$ is log-concave.  In fact, $a_{0}/\binom{n}{0},\ldots,a_{n}/\binom{n}{n}$ is log-concave.
\end{remark}

\begin{definition}
The \textbf{independence polynomial} of a graph $G$ on $n$ vertices is
$$I(G,t)\colonequals \sum_{k=0}^{n}x_{k}(G) t^{k},\qquad\forall\,t\in\R.$$
\end{definition}

\begin{theorem}[\cite{hamidoune90}]\label{hamthm}
If $G$ is claw-free, then $x_{0}(G),\ldots,x_{n}(G)$ is log-concave.
\end{theorem}
\begin{theorem}[\cite{chudnovsky07}]\label{chudthm}
If $G$ is claw-free, then $I(G,t)$ has all real roots.
\end{theorem}
Theorem \ref{chudthm} generalized Theorem \ref{hamthm} by Remark \ref{newtonrk}.

So, when $G$ is quite unlike a tree, its independent set sequence is log-concave and unimodal.  Unfortunately, the real-rooted property cannot hold for trees.  So, one cannot answer Question \ref{conj1} by proving a real-rootedness property.

\begin{example}[\embolden{Claw Graph}]
The tree on four vertices with one degree $3$ vertex has independence polynomial
$$1+4t+3t^{2}+t^{3},$$
which has two complex roots.
\end{example}

In contrast, the matching polynomial of a general graph behaves much better than the independence polynomial.  For a graph $G$ with $n$ vertices, let $m_{i}$ denote the number of matchings in $G$ with $i$ edges, and let $m_{0}\colonequals 1$.  Heilmann and Lieb \cite{in10} defined the matching polynomial of $G$ as
$$\mu_{G}(t)\colonequals\sum_{k\geq0}(-1)^{k}m_{k}t^{n-2k},\qquad\forall\,t\in\R.$$
\begin{theorem}[{\cite[Theorems 4.2 and 4.3]{in10}}]
For every graph $G$, $\mu_{G}(t)$ has only real roots.  Moreover, if $G$ has maximum degree $d$, then all roots of $\mu_{G}$ have absolute value at most $2\sqrt{d-1}$
\end{theorem}
This Theorem was used in the construction of bipartite Ramanujan graphs in \cite{in11}, resolving a Conjecture of Lubotzky.

We should also mention that independent sets and ``shattering'' phenomena have played a role in constructing a counterexample \cite{in12} for the weak Pinsker property \cite{in13} from ergodic theory.

\subsection{Our Contribution}

A \embolden{random tree} $T$ on $n\geq2$ vertices is a random graph that is equal to any of the $n^{n-2}$ possible labelled trees on $n$ vertices, each with probability $1/n^{n-2}$.

\begin{theorem}[\embolden{Main; Partial Unimodality for Random Trees}]\label{mainthm2}
There exists $c>0$ such that, with probability at least $1-e^{-cn}$, a random tree $T$ on $n$ vertices satisfies
$$x_{0}(T)<x_{1}(T)<\cdots<x_{\lfloor (.26543)n\rfloor}(T).$$
\end{theorem}
For comparison, the largest independent set size in a random tree is about $\approx.567143n$, with fluctuations of order $\sqrt{n}$, by the Azuma-Hoeffding inequality \cite{frieze90}.  So, the first $46.8\%$ of the nontrivial independent set sequence is unimodal.  Combined with Theorem \ref{levitthm} of Levit and Mandrescu, Question \ref{conj1} is ``four-fifths true'', with high probability.  As noted in Remark \ref{nearlyrk}, our argument does not currently recover Theorem \ref{levitthm} of Levit and Mandrescu, with high probability as $n\to\infty$.

An earlier version of the manuscript had the constant $.2$ in Theorem \ref{mainthm2}, though David Galvin pointed out that Lemma \ref{treelowerbound} (originally resembling Lemma \ref{lowerbound}), could be replaced with \cite[Theorem 5.1]{wingard95}, thereby improving the constant $.2$ to $.26543$ in Theorem \ref{mainthm2}.

Part of the difficulty of improving on Theorem \ref{mainthm2} and \ref{levitthm} is identifying a largest element of the sequence $x_{0}(T),\ldots,x_{n}(T)$.  Even for a random tree, it is unclear if a largest element of $x_{0}(T),\ldots,x_{n}(T)$ has index close to $(1/2)(.567)n$, or far from this value, with high probability.

The proof of Theorem \ref{mainthm2} has two main ingredients, based upon \cite{coja15}.  Let $k$ less than the expected size of the largest independent set in the graph.  If we condition on a uniformly random subset $S\subset\{1,\ldots,n\}$ of size $k$ being an independent set (i.e. size-biasing), then we need to show two things:

\begin{itemize}
\item[(1)] Probabilities of events do not change too much due to this conditioning (see Remarks \ref{pkrk}, \ref{rkconc}, Lemma \ref{lowerbound} and the improved Corollary \ref{betterERlowerbound}.)
\item[(2)] The number of vertices in $\{1,\ldots,n\}\setminus S$ that are connected to $S$ concentrates exponentially around its expected value, as in Chernoff bounds.  (See Lemma \ref{nchernoff}.)
\end{itemize}

In a random tree, the largest independent set is more than half of the size of the whole tree.  So, conditioning on a uniformly random subset $S\subset\{1,\ldots,n\}$ of size $k$ being an independent set changes the probability law of a uniformly random tree quite a bit.  This large change in measure is the reason Theorem \ref{mainthm2} can only apply for the first $35\%$ of the nontrivial independent set sequence of a random tree.

However, in e.g. an Erd\"{o}s-Renyi random graph of large degree $d$, the largest independent set in the graph is approximately a $(2/d)\log d$ fraction of the graph.  So, conditioning on a uniformly random subset $S\subset\{1,\ldots,n\}$ of size $k\leq (2/d)\log d$ being an independent set does not change the probability law of the random graph very much.  This small change in measure with high expected degree results in a proof of unimodality for essentially all of the independent set sequence of an Erd\"{o}s-Renyi random graph, as follows.

\begin{theorem}[\embolden{Second Main Theorem; Unimodality, Sparse Case, High Degree}]\label{thm1}
Let $\epsilon>0$.  Then for any $d\geq 10^{10/\epsilon}$, there exists $c>0$ such that, with probability at least $1-e^{-cn}$, $G\in G(n,d/n)$ satisfies
$$x_{0}(G)<x_{1}(G)<\cdots<x_{\lfloor\beta(1-\epsilon)/2\rfloor}(G)
,\quad\mathrm{and}\quad
x_{\lfloor\beta(1+\epsilon)/2\rfloor}(G)>\cdots >x_{\beta-1}(G)>x_{\beta}(G),
$$
where $\beta$ is the expected size of the largest independent set in $G(n,d/n)$.  (It is known that $\beta\approx (2/d)(\log d-\log\log d-\log 2+1)$ by Theorem \ref{friezethm} of \cite{frieze90}).
\end{theorem}

For the sake of demonstration, we write down the sub-optimal results when the expected degree $d$ is small.

\begin{theorem}[\embolden{Partial Unimodality, Sparse Case, Low Degree}]\label{thm2}
The independent set sequence of $G(n,d/n)$ is unimodal with high probability (as $n\to\infty$) for independent set sizes $k$ satisfying
\begin{itemize}
\item $k<.25n$ and $k>.46n$, when $d=1$.  (The largest independent set size is $\approx.728n$)
\item $k<.194n$ and $k>.39n$, when $d=2$.  (The largest independent set size is $\approx.607n$)
\item $k<.172n$ and $k>.35n$, when $d=e$.  (The largest independent set size is $\approx.552n$)
\end{itemize}
\end{theorem}

For random regular graphs, we can prove a result only applying to the increasing part of the independence set sequence.

\begin{theorem}[\embolden{Partial Unimodality, Sparse Regular Case, High Degree}]\label{thm5}
Let $\epsilon>0$.  For any $d\geq 10^{10/\epsilon}$, there exists $c>0$ such that, with probability at least $1-e^{-cn}$, if $G$ is a uniformly random $d$-regular random graph on $n$ vertices
$$x_{0}(G)<x_{1}(G)<\cdots<x_{\lfloor\beta(1-\epsilon)/2\rfloor}(G),$$
where $\beta$ is the expected size of the largest independent set in $G(n,d/n)$.
\end{theorem}

\subsection{Concurrent Results}

Abdul Basit and David Galvin \cite{basit20} are working on a result similar to Theorem \ref{mainthm2}, and they can also improve on the result of Levit and Mandrescu, Theorem \ref{levitthm}.

\subsection{Open Questions}

\begin{question}
Does log-concavity hold with high probability, for high-degree Erd\"{o}s-Renyi random graphs, as the number of vertices goes to infinity?  For example, can Theorem \ref{thm1} be improved to give log-concavity of the independent set sequence, rather than just unimodality?
\end{question}

\begin{question}
For uniformly random labelled trees, does the mode of the independent set sequence concentrate?  For example, let $T$ be a random tree on $n$ vertices, and let $M\colonequals\{m\in\{1,\ldots,n\}\colon x_{m}(T)=\max_{1\leq j\leq n}x_{j}(T)\}$.  Let $m$ be the largest element of $M$.  Does $m$ concentrate around any particular value, with high probability, as the number of vertices goes to infinity?
\end{question}

\begin{question}
Is there any exponential improvement to the lower bound of Lemma \ref{treelowerbound}?  That is, if we rewrite Lemma \ref{treelowerbound} as $X_{k,n}\geq c_{k,n}$, does there exist $b_{k}>0$ such that the inequality $X_{k,n}\geq e^{nb_{k}}c_{k,n}$ holds, with high probability as $n\to\infty$?  That is, is there a way to improve Theorem 5.1 from \cite{wingard95}?
\end{question}

\begin{question}
Is there a (deterministic) tree whose independent set sequence is \textit{not} log-concave?
\end{question}

\subsection{Notation}

\begin{itemize}
\item $X_{k,n}$ is the number of independent sets of size $k$ in an $n$ vertex graph.
\item $\mathcal{G}$ denotes a family of random graphs on $n$ vertices.
\item $\E$ and $\P$ denote the expected value with respect to the random graph $\mathcal{G}$.
\item $\Lambda_{k}(n)$ denotes the set of pairs $(G,\sigma)$ where $G\in\mathcal{G}_{n}$ and $\sigma\subset\{1,\ldots,n\}$ is an independent set in $G$ with cardinality $0\leq k\leq n$.
\item $\P_{\mathcal{U}_{k}(n)}$ denotes the uniform probability law on $\Lambda_{k}(n)$ (see Definition \ref{ukdef}).
\item $\P_{\mathcal{P}_{k}(n)}$ denotes the planted (size-biased) probability law on $\Lambda_{k}(n)$ (see Definition \ref{pkdef}).
\end{itemize}

\section{Preliminaries}

Let $V=\{1,\ldots,n\}$.  Let $0<p<1$.  Let $E$ be a random subset of $\{\{i,j\}\in V\times V\colon i\neq j\}$ such that
$$\P(\{i,j\}\in E)=p,\qquad\forall\,1\leq i<j\leq n$$
and such that the events $\{\{i,j\}\in E\}_{1\leq i<j\leq n}$ are independent.  Then $G=(V,E)$ is an \textbf{Erd\"{o}s-Renyi} random graph on $n$ vertices with parameter $0<p<1$.  This random graph is sometimes denoted as $G=G(n,p)$.

Let $k\geq0$ be an integer.   Recall that an independent set of $G$ is a subset of vertices such that no two of them are connected by an edge.  Let $X_{k,n}$ be the number of independent sets in $G$ of size $k$.

The second moment method shows that $X_{k,n}$ concentrates around its expected value when $p$ is larger than $n^{-1/2}(\log n)^{2}$.

\begin{lemma}[\cite{bollobas76} {\cite[Lemma 7.3]{janson11}}]
Assume that $p=p(n)$ satisfies
$$n^{-1/2}(\log n)^{2}\leq p\leq (\log n)^{-2}.$$
Then
$$\E X_{k,n}^{2}=(1+o(k^{2}/n))(\E X_{k,n})^{2}.$$
Consequently, we have
$$\P(\abs{X_{k,n}-\E X_{k,n}}>t(\E X_{k,n})(o(k/\sqrt{n})))\leq \frac{1}{t^{2}},\qquad\forall\,t>0.$$
\end{lemma}
So, when $p\geq n^{-1/2}(\log n)^{2}$, $X_{0,n},\ldots,X_{n,n}$ is unimodal in the sense that, for any $1\leq k\leq n$,
$$\P\Big(\frac{1-\epsilon}{1+\epsilon}\frac{\E X_{k+1,n}}{\E X_{k,n}}\leq\frac{X_{k+1,n}}{X_{k,n}}\leq \frac{1+\epsilon}{1-\epsilon}\frac{\E X_{k+1,n}}{\E X_{k,n}}\Big)\geq1-\frac{o(k^{2}/n)}{\epsilon^{2}},\quad\forall\,\epsilon>0.$$
More can be said when $p\geq n^{-1/3}(\log n)^{2}$.  Taking a union bound over $k$,
\begin{equation}\label{unionbdez}
\begin{aligned}
&\P\Big(\forall\,1\leq k\leq (2/p)\log(np),\quad \frac{1-\epsilon}{1+\epsilon}\frac{\E X_{k+1,n}}{\E X_{k,n}}\leq\frac{X_{k+1,n}}{X_{k,n}}\leq \frac{1+\epsilon}{1-\epsilon}\frac{\E X_{k+1,n}}{\E X_{k,n}}\Big)\\
&\qquad\qquad\qquad\qquad\qquad\geq1-\Big(\frac{2}{p}\log(np)\Big)^{3}\frac{o(1/n)}{\epsilon^{2}},\quad\forall\,\epsilon>0.
\end{aligned}
\end{equation}
So, setting $\beta\colonequals k/[(2/p)\log(np)]$ and using
\begin{equation}\label{ratioez}
\begin{aligned}
\frac{\E X_{k+1,n}}{\E X_{k,n}}
&=\frac{\binom{n}{k+1}(1-p)^{\binom{k+1}{2}}}{\binom{n}{k}(1-p)^{\binom{k}{2}}} =\frac{n-k}{k+1}(1-p)^{k}\\
&=\frac{n-k}{k+1}e^{-kp+o(kp)}=(1+o(1))\frac{n-k}{k+1}e^{-2\beta\log(np)}=(1+o(1))n(np)^{-2\beta},
\end{aligned}
\end{equation}
this quantity being larger than one or smaller than one tells us if the independent sequence is increasing or decreasing.  For example, if $p=n^{-1/4}$, then $n(np)^{-2\beta}=n^{1-3\beta/2}$, so \eqref{ratioez} goes to infinity when $\beta<2/3$ and it goes to zero when $\beta>2/3$.  That is, when $p=n^{-1/4}$, with high probability the first two thirds of the independent set sequence is increasing, and the last third is decreasing, by \eqref{unionbdez}.

The Second Moment Method fails to prove concentration of $X_{k,n}$ around its expected value when $p=O(n^{-1/2})$.  In this sparse case, the graph does not have many edges, and there are many independent sets in the graph that are highly correlated.  In the dense case when $p\gg n^{-1/2}$, there are fewer independent sets, and they have small correlations.  More specifically, we write
$$X_{k,n}=\sum_{S\subset V\colon \abs{S}=k}1_{\{S\,\,\mathrm{is}\,\,\mathrm{independent}\}}.$$
Taking the variance of both sides,
\begin{flalign*}
&\mathrm{Var}\Big(\sum_{S\subset V\colon \abs{S}=k}1_{\{S\,\,\mathrm{is}\,\,\mathrm{independent}\}}\Big)\\
&=\sum_{S\subset V\colon \abs{S}=k}\mathrm{Var}1_{\{S\,\,\mathrm{is}\,\,\mathrm{independent}\}}
+\sum_{\substack{S,T\subset V\colon\\ \abs{S}=\abs{T}=k,\,S\neq T}}\mathrm{Cov}\Big(1_{\{S\,\,\mathrm{is}\,\,\mathrm{independent}\}},1_{\{T\,\,\mathrm{is}\,\,\mathrm{independent}\}}\Big),
\end{flalign*}
where $\mathrm{cov}(X,Y)=\E ((X-\E X)(Y-\E Y))$.  In the uncorrelated case $p\gg n^{-1/2}$, the rightmost terms are small, so the variance of $X_{k,n}$ is smaller.  But in the correlated case $p=O(n^{-1/2})$, the rightmost terms are large, so the variance of $X_{k,n}$ is larger.

Nevertheless, the \textit{size} of the largest independent set in the Erd\"{o}s-Renyi random graph does concentrate around its expected value by the Azuma-Hoeffding inequality (Lemma \ref{azumah}) or Talagrand's convex distance inequality (Theorem \ref{talagrand}); see Theorem \ref{friezethm} below.

\begin{theorem}[\embolden{Talagrand's Convex Distance Inequality}, {\cite[Theorem 2.29]{janson11}}]\label{talagrand}
Let $Z_{1},\ldots,Z_{n}$ be independent random variables taking values in $\Gamma_{1},\ldots,\Gamma_{n}$, respectively.  Let $f\colon \Gamma_{1}\times\cdots\times\Gamma_{n}\to\R$.  Let $X\colonequals f(Z_{1},\ldots,Z_{n})$.  Suppose there are constants $c_{1},\ldots,c_{n}\in\R$ and $\psi\colon\R\to\R$ such that
\begin{itemize}
\item Let $1\leq k\leq n$.  If $z,z'\in\Gamma\colonequals\prod_{i=1}^{n}\Gamma_{i}$ differ in only the $k^{th}$ coordinate, then $\abs{f(z)-f(z')}\leq c_{k}$.
\item If $z\in\Gamma$ and $r\in\R$ satisfy $f(z)\geq r$, then there exists a ``certificate'' $J\subset\{1,\ldots,n\}$ with $\sum_{i\in J}c_{i}^{2}\leq\psi(r)$ such that, for all $y\in\Gamma$ with $y_{i}=z_{i}$ for all $i\in J$, we have $f(y)\geq r$.
\end{itemize}
If $m$ is a median for $X$, then for every $t>0$,
$$\P(X\leq m-t)\leq 2e^{-t^{2}/(4\psi(m))}.$$
$$\P(X\geq m+t)\leq 2e^{-t^{2}/(4\psi(m+t))}.$$
\end{theorem}

\subsection{Estimates for the Independent Set of Largest Size}

Let $V=\{1,\ldots,n\}$.  Let $E$ be a random subset of $\{\{i,j\}\in V\times V\colon i\neq j\}$.  Then $G=(V,E)$ is a random graph on $n$ vertices.  Fix $0\leq m\leq n$.  Let $\mathcal{F}_{m}$ denote the set of subgraphs of $V$ with the equivalence relation defined so that $A,B\in\mathcal{F}_{m}$ are equivalent if and only if: for all $1\leq i\leq m$, for all $1\leq j\leq n$, $\{i,j\}\in A$ if and only if $\{i,j\}\in B$.  Then $\mathcal{F}_{0}\subset\mathcal{F}_{1}\cdots\subset\mathcal{F}_{n}$.

Let $Y\colon G\to\R$ be a function.  For any $1\leq m\leq n$, define
$$Y_{m}\colonequals \E (Y|\mathcal{F}_{m}).$$
Then $Y_{0},Y_{1},\ldots,Y_{n}$ is a martingale, $Y_{0}=\E Y$ and $Y_{n}=Y$.

In the case that $G$ is an Erd\"{o}s-Renyi random graph, and $Y$ is the size of the largest independent set in the graph, we have $\abs{Y_{m+1}-Y_{m}}\leq 1$ for all $0\leq m\leq n-1$, and we can use the following Lemma.

\begin{lemma}[Azuma-Hoeffding Inequality {\cite{shamir87}}]\label{azumah}
Let $Y_{0},\ldots,Y_{n}$ be a martingale with $Y_{0}$ constant and $\abs{Y_{m+1}-Y_{m}}\leq c$ for all $0\leq m\leq n-1$.  Then
$$\P(\abs{Y_{n}-Y_{0}}>t)\leq2e^{-\frac{t^{2}}{2c^{2}n}},\qquad\forall\,t>0.$$
\end{lemma}

The following Theorem follows from Lemma \ref{azumah} inequality (and some additional arguments).

\begin{theorem}[{\cite{frieze90}}]\label{friezethm}
Let $Y_{n}$ denote the size of the largest independent set of an Erd\"{o}s-Renyi random graph on $n$ vertices with parameter $0<p<1$.  Let $d\colonequals np$ and let $\epsilon>0$ both be fixed.  Then $\exists$ $d_{\epsilon}$ a suitably large constant such that $d_{\epsilon}\leq d$ such that the following holds.  With probability $1$ as $n\to\infty$, we have
$$\abs{X_{n,p}-\frac{2n}{d}(\log d-\log\log d-\log 2+1)}\leq\frac{\epsilon n}{d}.$$
\end{theorem}
In fact, this inequality is violated with probability at most $2\exp(-\frac{\epsilon^{2}nd}{2(\log d)^{2}})$.  Also, we may choose $d_{\epsilon}\colonequals 10^{10/\epsilon}$.
%

Theorem \ref{friezethm} gives an explicit concentration of $Y_{n}$ around its expected value when the expected degree of the Erd\"{o}s-Renyi random graph has sufficiently high expected degree.  The case that the expected degree is close to $1$ follows a different argument.

\begin{lemma}[{\cite[Theorem 1; Corollary p. 375; for $G(n,p)$]{karp81}}]\label{lowdegreeindset}
Let $d\in\R$ satisfy $0<d<e$.  Let $n$ be a positive integer. 
The expected size of the largest matching in $G(n,d/n)$ is
$$n\Big(1-\frac{a+b+ab}{2d}\Big)+o(n).$$
Here $a$ is the smallest solution of the equation $x=de^{-de^{-x}}$ and $b=de^{-a}$.  So (as noted in \cite[Corollary 1]{gamarnik06}), the expected size of the largest independent set in $G(n,d/n)$ is
$$n\frac{a+b+ab}{2d}+o(n).$$
\end{lemma}

\begin{example}
When $d=1$, we have $a\approx .567$, and $b\approx .567$, and $\E Y_{n}=n(.272 +o(1))$.
\end{example}
\begin{example}
When $d=2$, we have $a\approx .852$, and $b\approx .853$ and $\E Y_{n}=n(.393+o(1))$.
\end{example}
\begin{example} 
When $d=e$, we have $a=b=1$, and $\E Y_{n}= n(.448+o(1))$.
\end{example}

\begin{remark}\label{rk1}
In a bipartite graph, the size $\nu$ of the maximum matching is equal to the size $\beta$ of the minimum vertex cover, by K\"{o}nig's Theorem. In any $n$-vertex graph, we have $\alpha+\beta=n$, where $\alpha$ is the size of the largest independent set.  A tree is a bipartite graph.  So, in a tree, we have $\alpha=n-\nu$.

In general, $\beta\geq\nu$ (since the minimum vertex cover must contain at least one vertex from each matched edge of the maximum matching), so $\alpha\leq n-\nu$ in general.
\end{remark}

Let $Y_{n}$ be the size of the largest independent set in a random labelled tree on $n$ vertices.  It follows from \cite[p. 278, p. 282]{meir88} (see also \cite[Theorem 4.7]{bhamidi12}) that, as $n\to\infty$, $Y_{n}/n$ converges in probability to the constant
$$.567\ldots$$
More generally, we have
\begin{theorem}[\embolden{Central Limit Theorem for Largest Independent Set of Random Trees}, {\cite{pittel99}}]\label{pittelthm}
Let $\rho$ satisfy $\rho e^{\rho}=1$, so that $\rho\approx.567143$.  Define
$$\sigma\colonequals\rho(1-\rho-\rho^{2})(1+\rho)^{-1}.$$
Let $Y_{n}$ denote the largest size of an independent set in a uniformly random labelled connected tree on $n$ vertices.  Then
$$\E Y_{n}=\rho n+\frac{\rho^{2}(\rho+2)}{2(\rho+1)^{3}}+O(n^{-1}),
\qquad\mathrm{Var}(Y_{n})=\sigma^{2}n+O(1),$$ and
$$\frac{Y_{n}-\E Y_{n}}{\sqrt{\mathrm{Var}(Y_{n})}}$$
converges in distribution to a standard Gaussian random variable as $n\to\infty$.
\end{theorem}

\section{Concentration Bounds and the Planted Model}

We assume there is a probability law $\P$ on $\mathcal{G}_{n}$ that defines the probability that the graph $G\in\mathcal{G}_{n}$ occurs.  We let $\E$ denote the expected value with respect to $\P$.

\begin{definition}\label{lambdadef}
Let $\Lambda_{k}(n)$ be the set of pairs $(G,\sigma)$ of graphs $G\in\mathcal{G}_{n}$ together with independent sets $\sigma$ of vertices of size $k$.  Let $X_{k,n}(G)$ be the number of independent sets of size $k$ in the graph $G$.
\end{definition}

\begin{definition}[\embolden{The Distribution $\mathcal{U}_{k}(n)$}]\label{ukdef}
The distribution $\mathcal{U}_{k}(n)$ is defined on $\Lambda_{k}(n)$ as follows.
\begin{itemize}
\item Let $G$ be a graph in $\mathcal{G}_{n}$, drawn with probability $\P$.
\item Let $\sigma$ be chosen uniformly at random among all independent sets in $G$ of size $k$ (if it exists).
\item Output the pair $(G,\sigma)$.  (If no such $\sigma$ exists, output nothing.)
\end{itemize}
\end{definition}

\begin{definition}[\embolden{The Distribution $\mathcal{P}_{k}(n)$, or the ``Planted Model''}]\label{pkdef}
The distribution $\mathcal{P}_{k}(n)$ is the distribution $\mathcal{U}_{k}(n)$, conditioned on a uniformly random subset of size $k$ being independent.  That is, the distribution $\mathcal{P}_{k}(n)$ is defined by
\begin{itemize}
\item Selecting a subset $\sigma$ of $\{1,\ldots,n\}$ of size $k$ uniformly at random,
\item selecting a graph $G\in\mathcal{G}_{n}$ in which $\sigma$ is an independent set, with probability\\\qquad $\P(G|\sigma\,\,\mathrm{is}\,\,\mathrm{independent}\,\,\mathrm{in}\,\, G)=\P(G)/\sum_{\{G'\in\mathcal{G}_{n}\colon \sigma\,\,\mathrm{is}\,\,\mathrm{independent}\,\,\mathrm{in}\,\, G'\}}\P(G')$,
\item then outputting the pair $(G,\sigma)$.
\end{itemize}
The distribution $\mathcal{P}_{k}(n)$ could also be called the \textbf{size-biased} distribution.
\end{definition}

\begin{lemma}\label{pkrk}
Let $\mathcal{G}=\mathcal{G}_{n}$ be a family of random graphs on the labelled vertices $\{1,\ldots,n\}$ that is invariant with respect to permutations of the vertices.  That is, $\forall$ $G\in\mathcal{G}_{n}$, $\P(G)=\P(G')$ where $G'$ is $G$ with its vertices permuted by an arbitrary permutation of $\{1,\ldots,n\}$.  Then

$$
\P_{\mathcal{P}_{k}(n)}(G,\sigma)
=\frac{\P(G)}{\E X_{k,n}},\qquad\forall\,(G,\sigma)\in\Lambda_{k}(n).
$$
$$
\P_{\mathcal{U}_{k}(n)}(G,\sigma)
=\frac{\P(G)}{ X_{k,n}(G)},\qquad\forall\,(G,\sigma)\in\Lambda_{k}(n).
$$
\end{lemma}
\begin{remark}
If $(G,\sigma)\in\Lambda_{k}(n)$ exists, then $\E X_{k,n}>0$ and $X_{k,n}>0$, so a division by zero does not occur.
\end{remark}
\begin{proof}
The first identity follows from Definition \ref{ukdef}.  Now, by Definition \ref{pkdef},
$$
\P_{\mathcal{P}_{k}(n)}(G,\sigma)=\frac{\P(G)}
{\binom{n}{k}
\sum_{\{G'\in\mathcal{G}_{n}\colon \widetilde{\sigma}\,\,\mathrm{is}\,\,\mathrm{independent}\,\,\mathrm{in}\,\, G'\}}\P(G')}.
$$
Also,
\begin{flalign*}
\E X_{k,n}
&=\sum_{\widetilde{\sigma}\subset\{1,\ldots,n\}\colon\absf{\widetilde{\sigma}}=k}\P(\sigma\,\,\mathrm{is}\,\,\mathrm{an}\,\,\mathrm{independent}\,\,\mathrm{set})\\
&=\sum_{\widetilde{\sigma}\subset\{1,\ldots,n\}\colon\absf{\widetilde{\sigma}}=k}\,\,\sum_{\{G'\in\mathcal{G}_{n}\colon\widetilde{\sigma}\,\,\mathrm{is}\,\,\mathrm{independent}\,\,\mathrm{in}\,\, G'\}}\P(G').
\end{flalign*}
So,
\begin{flalign*}
\E X_{k,n}\cdot \P_{\mathcal{P}_{k}(n)}(G,\sigma)
&=\P(G)\frac{\sum_{\widetilde{\sigma}\subset\{1,\ldots,n\}\colon\absf{ \widetilde{\sigma}}=k}\,\,\sum_{\{G'\in\mathcal{G}_{n}\colon \widetilde{\sigma}\,\,\mathrm{is}\,\,\mathrm{independent}\,\,\mathrm{in}\,\, G'\}}\P(G')}{\binom{n}{k}
\sum_{\{G'\in\mathcal{G}_{n}\colon \sigma\,\,\mathrm{is}\,\,\mathrm{independent}\,\,\mathrm{in}\,\, G'\}}\P(G')}\\
&=\P(G)\frac{\E_{\widetilde{\sigma}}\P(\widetilde{\sigma}\,\,\mathrm{is}\,\,\mathrm{independent})}
{\P(\sigma\,\,\mathrm{is}\,\,\mathrm{independent})}.
\end{flalign*}
In the permutation invariant case, the last fraction is one.
\end{proof}
\begin{remark}\label{rkconc}
So, $X_{k,n}$ concentrates around its expected value \textit{if and only if} $\P_{\mathcal{P}_{k}(n)}$ and $\P_{\mathcal{U}_{k}(n)}$ are comparable probability measures.
\end{remark}
\begin{remark}
Without the permutation invariance assumption, we have
$$
\P_{\mathcal{U}_{k}(n)}(G,\sigma)
=\frac{\P(G)}{ X_{k,n}(G)},\qquad\forall\,(G,\sigma)\in\Lambda_{k}(n).
$$
$$
\P_{\mathcal{P}_{k}(n)}(G,\sigma)
=\frac{\P(G)}{\E X_{k,n}}\frac{\E_{\widetilde{\sigma}}\P(\widetilde{\sigma}\,\,\mathrm{is}\,\,\mathrm{independent})}
{\P(\sigma\,\,\mathrm{is}\,\,\mathrm{independent})},\qquad\forall\,(G,\sigma)\in\Lambda_{k}(n).
$$
Here $\E_{\widetilde{\sigma}}$ denotes the uniform probability law over all subset of $\{1,\ldots,n\}$ of size $k$.
\end{remark}

\subsection{Change of Measure}

In this section, we give a comparison between the uniform and size biased measures.

\begin{lemma}\label{lemma40a}
Let $A\subset\Lambda_{k}(n)$.  Let $\mathcal{G}$ be a family of random graphs that is invariant under permutations of the vertices $\{1,\ldots,n\}$.  Let $c>0$.  Let $C\subset\Lambda_{k}(n)$ be the set of $(G,\sigma)$ such that $X_{k,n}\geq c\E X_{k,n}$.  Then
\begin{equation}\label{three0}
\P_{\mathcal{U}_{k}(n)}(A)\leq\frac{1}{c}\P_{\mathcal{P}_{k}(n)}(A\cap C)+\P_{\mathcal{U}_{k}(n)}(C^{c}).
\end{equation}
\end{lemma}
\begin{proof}%
By Lemma \ref{pkrk} and by definition of $C$,
\begin{equation}\label{supp}
\P_{\mathcal{U}_{k}(n)}(A\cap C)=\sum_{(G,\sigma)\in A\cap C}\frac{\P(G)}{X_{k,n}(G)}
\leq\frac{1}{c}\sum_{(G,\sigma)\in A\cap C}\frac{\P(G)}{\E X_{k,n}}
=\frac{1}{c}\P_{\mathcal{P}_{k}(n)}(A\cap C).
\end{equation}
Therefore,
$$
\P_{\mathcal{U}_{k}(n)}(A)
=\P_{\mathcal{U}_{k}(n)}(A\cap C)+\P_{\mathcal{U}_{k}(n)}(A\cap C^{c})
\stackrel{\eqref{supp}}{\leq}\frac{1}{c}\P_{\mathcal{P}_{k}(n)}(A\cap C)+\P_{\mathcal{U}_{k}(n)}(C^{c}).
$$
\end{proof}
\begin{remark}
Without the permutation invariance assumption, we have
$$\P_{\mathcal{U}_{k}(n)}(A\cap C)\leq\frac{1}{c}\P_{\mathcal{P}_{k}(n)}(A\cap C)\frac{\E_{\widetilde{\sigma}}\P(\widetilde{\sigma}\,\,\mathrm{is}\,\,\mathrm{independent})}
{\min_{\sigma\colon (G,\sigma)\in A\cap C}\P(\sigma\,\,\mathrm{is}\,\,\mathrm{independent})},$$
$$ \P_{\mathcal{U}_{k}(n)}(C)\leq  \epsilon+\frac{1}{c}\P_{\mathcal{P}_{k}(n)}(A)\frac{\E_{\widetilde{\sigma}}\P(\widetilde{\sigma}\,\,\mathrm{is}\,\,\mathrm{independent})}
{\min_{\sigma\colon (G,\sigma)\in A\cap C}\P(\sigma\,\,\mathrm{is}\,\,\mathrm{independent})}.$$

\end{remark}

\subsection{A Lower Bound}

Applying Lemma \ref{lemma40a} requires a lower bound on $X_{k,n}$ that holds with high probability.  The following elementary lower bound is suitable for small degree graphs, but for large degree graphs we require an even sharper bound in Section \ref{sectal}.

\begin{lemma}[\embolden{Lower Bound}]\label{lowerbound}
Let $\beta n$ be the expected size of the largest size independent set in $G(n,d/n)$.
Then, for all $t>0$, with probability at least $1-e^{-t^{2}/(2n)}$,
$$X_{k,n}\geq e^{-t\log(1-\alpha/\beta)}\sqrt{\frac{\beta(1-\alpha)}{\beta-\alpha}}e^{d\alpha}e^{-n[(\beta-\alpha)\log(\beta-\alpha)-\beta\log\beta-(1-\alpha)\log(1-\alpha)-d\alpha^{2}/2]}\E X_{k,n}.$$
\end{lemma}
\begin{proof}
Note that
\begin{equation}\label{three1}
\E X_{k,n}=\binom{n}{\alpha n}(1-d/n)^{\binom{\alpha n}{2}}.
\end{equation}
Also, the size of the largest independent set concentrates around it's expected value by the Azuma-Hoeffding inequality \ref{azumah}. A lower bound for $X_{k,n}$ is given by taking size $k=\alpha n$ subsets of the largest independent set, i.e.
$$\P\left(X_{k,n}\geq \binom{\beta n-t}{\alpha n}\right)\geq e^{-t^{2}/[2n]},\qquad\forall\,t\geq0.$$
Therefore, by Stirling's formula and $(1-x)\leq e^{-x}$ for all $x\in\R$,
\begin{flalign*}
\frac{\binom{n}{\alpha n}}{\binom{\beta n}{\alpha n}}(1-d/n)^{\binom{\alpha n}{2}}
&=\frac{n![(\beta-\alpha)n]!}{(\beta n)![(1-\alpha)n]!}(1-d/n)^{\binom{\alpha n}{2}}\\
&=\sqrt{\frac{\beta-\alpha}{\beta(1-\alpha)}}e^{-d\alpha/2}e^{n[(\beta-\alpha)\log(\beta-\alpha)-\beta\log\beta-(1-\alpha)\log(1-\alpha)-d\alpha^{2}/2]}(1+o_{n}(1)).
\end{flalign*}
More specifically,
\begin{equation}\label{three2}
\begin{aligned}
&\frac{\binom{n}{\alpha n}}{\binom{\beta n-t}{\alpha n}}(1-d/n)^{\binom{\alpha n}{2}}\\
&\leq\sqrt{\frac{\beta-\alpha}{\beta(1-\alpha)}}e^{-d\alpha/2}e^{n[(\beta-\alpha)\log(\beta-\alpha)-\beta\log\beta-(1-\alpha)\log(1-\alpha)-d\alpha^{2}/2]}e^{t\log(1-\alpha/\beta)}(1+o_{n}(1)).
\end{aligned}
\end{equation}
That is, with probability at least $1-e^{-t^{2}/(2n)}$,
\begin{flalign*}
\frac{\E X_{k,n}}{X_{k,n}}
&\leq \frac{\E X_{k,n}}{\#\,\mbox{size $k$ subsets of the largest independent set}}\\
&\stackrel{\eqref{three1}\wedge\eqref{three2}}{\leq}\sqrt{\frac{\beta-\alpha}{\beta(1-\alpha)}}e^{-d\alpha/2}e^{n[(\beta-\alpha)\log(\beta-\alpha)-\beta\log\beta-(1-\alpha)\log(1-\alpha)-d\alpha^{2}/2]}
e^{t\log(1-\alpha/\beta)}.
\end{flalign*}
\end{proof}

\section{Talagrand's Inequality and Lower Bounds}\label{sectal}

The lower bound of Lemma \ref{lowerbound} is not precise enough to prove full unimodality of the independent set sequence, since it only uses the Azuma-Hoeffding inequality.  A more precise lower bound appears in \cite[Proposition 22]{coja15} using the more precise Talagrand convex distance inequality, Theorem \ref{talagrand}.  Due to a few small typos, and focusing on $G(n,m)$ as opposed to $G(n,p)$, we present the argument of \cite{coja15} below.

\begin{theorem}[{\cite[Theorem 1]{dani11}}]\label{danithm}
For any $0<\alpha<1$, define
\begin{equation}\label{danieqn}
d'(\alpha)\colonequals\sup\{d>0\colon\lim_{n\to\infty}\P(X_{\alpha n,n}(G(n,d/n))>0)=1\}.
\end{equation}
Then, for all $0<\alpha<10^{-9}$,
$$d'(\alpha)\geq 2\frac{\log(1/\alpha)+1}{\alpha}-\frac{2}{\sqrt{\alpha}}.$$
\end{theorem}

Talagrand's large deviation inequality then implies

\begin{theorem}[{\cite[Theorem 14]{coja15}}]\label{misco}
Define $d'(\alpha)$ by \eqref{danieqn}.  Let $d\geq d'(\alpha)$ and let $0<\alpha<10^{-9}$.  Let $Y_{n}$ be the size of the largest independent set in the graph $G(n,d/n)$.  Then
$$\P(Y_{n}(G(n,d/n))<t)\leq 4e^{-\frac{(\alpha n-t+1)^{2}}{4\alpha n}},\qquad\forall\,t<\alpha n.$$
\end{theorem}
\begin{cor}[{\cite[Corollary 15]{coja15}}]\label{cojacor}
Let $0<\alpha<10^{-4}$.  Define
\begin{equation}\label{ddef}
d\colonequals2\frac{\log(1/\alpha)+1}{\alpha}-\frac{8}{\sqrt{\alpha}}.
\end{equation}
Then $\exists$ $c>0$ such that
$$\P(Y_{n}(G(n,d/n))<\alpha n)\leq 4e^{-\frac{n}{d^{2}(\log d)^{4}}},\qquad\forall\,n\geq1.$$
$$\E X_{\alpha n,n}\leq c\cdot e^{10 n(\log d)^{3/2}d^{-3/2}},\qquad\forall\,n\geq1.$$
\end{cor}
\begin{proof}   
%
Let $\alpha<\alpha'<2\alpha$ such that $d=2\frac{\log(1/\alpha')+1}{\alpha'}-\frac{2}{\sqrt{\alpha'}}$.  By Theorem \ref{misco} (applied for the parameters $d$ and $\alpha'$),
$$\P(Y_{n}(G(n,d/n))<\alpha n)\leq 4e^{-\frac{([\alpha'-\alpha] n+1)^{2}}{4\alpha' n}}
\leq4e^{-n\frac{(\alpha'-\alpha)^{2}}{8\alpha }}.$$
A calculation as in \cite[Corollary 4]{coja15} shows that
$$\alpha'-\alpha\geq 10\alpha/\sqrt{d(\log d)^{5}}.$$
Also, $\alpha>1.5(\log d)/d$, so
$$\P(Y_{n}(G(n,d/n))<\alpha n)\leq4e^{-n\alpha\frac{100}{8d(\log d)^{5}}}
\leq4e^{-n\frac{150}{8d^{2}(\log d)^{4}}}.$$
Meanwhile, by Stirling's formula,

\begin{flalign*}  %
\E X_{\alpha n ,n}
&=\binom{n}{\alpha n}(1-d/n)^{\binom{\alpha n}{2}}
\leq\frac{1}{\sqrt{\alpha(1-\alpha)}} e^{n[-\alpha\log\alpha-(1-\alpha)\log(1-\alpha)-(d/2)\alpha^{2}]}e^{-d\alpha}(1+o_{n}(1))\\
&\stackrel{\eqref{ddef}}{=} e^{n[-\alpha\log\alpha-(1-\alpha)\log(1-\alpha)-\alpha(\log(1/\alpha)+1)+4\alpha^{3/2}]}\frac{e^{2(\log(1/\alpha)+1)-8\alpha^{1/2}}}{\sqrt{\alpha(1-\alpha)}} (1+o_{n}(1))\\
&\approx e^{n[-(1-\alpha)\log(1-\alpha)-\alpha+4\alpha^{3/2}]}c(\alpha)(1+o_{n}(1))\\
&\approx e^{n[-(1-\alpha)[-\alpha+O(\alpha^{2})]-\alpha+4\alpha^{3/2}]}c(\alpha)(1+o_{n}(1))\\
&\approx e^{n[4\alpha^{3/2}+O(\alpha^{2})]}\approx e^{4n[\alpha^{3/2}+O(\alpha^{2})]}c(\alpha)(1+o_{n}(1)).
\end{flalign*}
And $\alpha\leq 2(\log d)/d$ by (inverting) the assumption \eqref{ddef}, so the conclusion follows.
\end{proof}

\begin{lemma}[{\cite[Lemma 23]{coja15}}]\label{ublem}
Let $d'>d$.  Let $0<k<n$.  Then
$$\P\Big(X_{k,n}(G(n,d/n))<\frac{\E X_{k,n}(G(n,d/n)) }{2\E X_{k,n}(G(n,d'/n))}\Big)\leq2\P(X_{k,n}(G(n,d'/n))=0).$$
\end{lemma}
\begin{proof}  
By direct calculation,
\begin{equation}\label{four4}
\E X_{k,n}(G(n,d/n))=\binom{n}{k}(1-d/n)^{\binom{k}{2}},\qquad \E X_{k,n}(G(n,d'/n))=\binom{n}{k}(1-d'/n)^{\binom{k}{2}}.
\end{equation}  
Let $d''>0$ and note that $(1-d/n)(1-d''/n)=(1-(d+d'')/n+d[d'']/n^{2})$.  So, if we define $d''\colonequals [d'-d]/(1+d/n)$, then $G(n,d')$ can be described by the following two-step procedure.  The probability that any edge exists in $G(n,d'/n)$ is given by first constructing $G(n,d/n)$, and then keeping every edge in $G(n,d/n)$ with probability $d''/n$.  By \eqref{four4},
\begin{equation}\label{four5}
\E \Big[X_{k,n}(G(n,d'/n))\Big| X_{k,n}(G(n,d/n))\Big]=X_{k,n}(G(n,d/n))(1-d''/n)^{\binom{k}{2}}.
\end{equation}
Let $C$ be the event that
$$X_{k,n}(G(n,d/n))<\frac{\E X_{k,n}(G(n,d/n))}{2\E X_{k,n}(G(n,d'/n))}.$$
By Markov's inequality,
\begin{flalign*}
&\frac{1}{2}\leq\P\Big(X_{k,n}(G(n,d'/n))<2\E [X_{k,n}(G(n,d'/n))| C]\Big| C\Big)\\
&\qquad\stackrel{\eqref{four5}}{\leq}
\P\Big(X_{k,n}(G(n,d'/n))<2\frac{(1-d''/n)^{\binom{k}{2}}\E X_{k,n}(G(n,d/n))}{\E X_{k,n}(G(n,d'/n))}\Big| C\Big).
\end{flalign*}
So, by definition of conditional probability and \eqref{four5},
$$\P(C)\leq 2\P(X_{k,n}(G(n,d'/n))<1)=2\P(X_{k,n}(G(n,d'/n))=0).$$
\end{proof}

\begin{cor}[\embolden{Lower Bound}, {\cite[Proposition 22]{coja15}}]\label{betterERlowerbound}
There exists $\epsilon_{d}\colonequals \log\log d/\log d$ such that, for any $k<(2-\epsilon_{d})n(\log d)/d$,
$$\P\Big(X_{k,n}(G(n,d/n))\geq e^{-20nd^{-3/2}(\log d)^{-3/2}}\E X_{k,n}(G(n,d/n))\Big)\geq 1-8\exp\Big(-\frac{n}{2d^{2}(\log d)^{2}}\Big).$$
\end{cor}
\begin{proof}
Define $k$ such that
\begin{equation}\label{three3}
k\leq(2n/d)(\log d-\log\log d+1-\log 2),\qquad \alpha\colonequals k/n,
\end{equation}
\begin{equation}\label{three4}
\widetilde{d}\colonequals2(-\log\alpha+1)/\alpha  -8/\sqrt{\alpha}.
\end{equation}

From Corollary \ref{cojacor}, (applied to the parameters $\alpha,\widetilde{d}>d$),
\begin{equation}\label{four1}
\begin{aligned}
\P(Y_{n}(G(n,\widetilde{d}/n))<\alpha n)&\leq 4\exp\Big(-\frac{n}{\widetilde{d}^{2}(\log \widetilde{d})^{4}}\Big).\\
\E X_{\alpha n,n}(G(n,\widetilde{d}/n))&\leq c\cdot e^{10 n(\log \widetilde{d})^{3/2}\widetilde{d}^{-3/2}}.
\end{aligned}
\end{equation}
We now compare $d$ and $\widetilde{d}$.  We have

\begin{equation}\label{three5}
\begin{aligned}
-\log \alpha+1
&\stackrel{\eqref{three3}}{=}\log d-\log\log d+1-\log 2-\log\Big(1-\frac{\log\log d-1+\log 2}{\log d}\Big)\\
&=  \log d-\log\log d+1-\log 2+\frac{\log\log d-1+\log 2}{\log d}+O\Big(\frac{\log\log d}{\log d}\Big)^{2}.
\end{aligned}
\end{equation}
So
\begin{flalign*}
\widetilde{d}
&\stackrel{\eqref{three4}}{=}\frac{2}{\alpha}(-\log\alpha+1)+8/\sqrt{\alpha}\\
&\stackrel{\eqref{three3}\wedge\eqref{three5}}{=}\frac{d}{ \log d-\log\log d+1-\log 2}\cdot\Big(\log d-\log\log d+1-\log 2\\
&\qquad\qquad+\frac{\log\log d-1+\log 2}{\log d}+O\Big(\frac{\log\log d}{\log d}\Big)^{2}\Big)+8/\sqrt{\alpha}\\
&=d\Big[1+\frac{\log\log d}{(\log d)^{2}}+O\Big(\frac{(\log\log d)^{2}}{(\log d)^{3}}\Big)\Big]+O(\sqrt{d}).
\end{flalign*}

So, as $d\to\infty$, $\widetilde{d}>d$ are asymptotically comparable as $d\to\infty$, so \eqref{four1} implies that
\begin{equation}\label{four8}
\begin{aligned}
\P(Y_{n}(G(n,d/n))<\alpha n)
&\leq\P(Y_{n}(G(n,\widetilde{d}/n))<\alpha n)\\
&\leq 4\exp\Big(-\frac{n}{\widetilde{d}^{2}(\log \widetilde{d})^{4}}\Big)
\leq 4\exp\Big(-\frac{n}{2d^{2}(\log d)^{2}}\Big).
\end{aligned}
\end{equation}
\begin{equation}\label{four6}
\E X_{\alpha n,n}(G(n,\widetilde{d}/n))
\leq c\cdot e^{10 n(\log \widetilde{d})^{3/2}\widetilde{d}^{-3/2}}
\leq c\cdot e^{10 n(\log d)^{-3/2}d^{-3/2}}.
\end{equation}

Now, note that the following two events are equal:
\begin{equation}\label{four7}
\{Y_{n}<\alpha n\}=\{X_{\alpha n,n}=0\}.
\end{equation}
(Recall that $Y_{n}$ is the size of the largest independent set and $X_{\alpha n,n}$ is the number of independent sets of size $\alpha n$.)  Lemma \ref{ublem} then implies that
\begin{flalign*}
&\P\Big(X_{\alpha n,n}(G(n,d/n))<\frac{1}{2}c\cdot e^{10 n(\log d)^{-3/2}d^{-3/2}}\E X_{\alpha n,n}(G(n,d/n))\Big)\\
&\stackrel{\eqref{four6}}{\leq}\P\Big(X_{\alpha n,n}(G(n,d/n))<\frac{\E X_{\alpha n,n}(G(n,d/n))}{2\E X_{\alpha n,n}(G(n,\widetilde{d}/n))}\Big)\\
&\leq 2\P(X_{\alpha n,n}(G(n,\widetilde{d}/n))=0)
\stackrel{\eqref{four7}}{=}2\P(Y_{n}(G(n,\widetilde{d}/n))<\alpha n)
\stackrel{\eqref{four8}}{\leq} 8\exp\Big(-\frac{n}{2d^{2}(\log d)^{2}}\Big).
\end{flalign*}
\end{proof}

\section{Unimodality Arguments: Abstract Version}

\begin{lemma}[\embolden{Counting Lemma}]\label{countinglemma}
Let $G=(V,E)$ be any (deterministic) graph.  Then
$$\sum_{\substack{S\subset\{1,\ldots,n\}\colon\\\abs{S}=k}}1_{\{S\,\,\mathrm{is}\,\,\mathrm{independent}\}}
\cdot\#\{\mathrm{vertices}\,\,\mathrm{not}\,\,\mathrm{connected}\,\,\mathrm{to}\,\,S\}=(k+1)X_{k+1,n}.$$
\end{lemma}
\begin{proof}
Each side is found by counting the number of ordered pairs $(S,T)$ of independent sets such that $S\subset T$, where $\abs{S}=k$ and $\abs{T}=k+1$.  These pairs can be counted by first starting with any independent set $S$ of size $k$, and then adding a single vertex to it.  This count gives the left side.  On the right side, we can alternatively start with an independent set $T$ of size $k+1$, and then delete a single vertex from it, in $k+1$ different ways.
\end{proof}

Let $(G,\sigma)\subset\Lambda_{k}(n)$.  Let $N=N_{\sigma}$ be the number of vertices in $G\setminus\sigma$ not connected to $\sigma$.  Let $s,s'>0$.  Let $A$ be the event that $(G,\sigma)$ satisfies $$\{N_{\sigma}-\E_{\mathcal{P}_{k}(n)} N_{\sigma}<-s\E_{\mathcal{P}_{k}(n)}  N_{\sigma}\}\cup\{N_{\sigma}-\E_{\mathcal{P}_{k}(n)} N_{\sigma}>s'\E_{\mathcal{P}_{k}(n)}  N_{\sigma}\}.$$
\begin{lemma}[\embolden{Ratio Bound}]\label{ratiolemma}
Let $c,\epsilon_{1},\epsilon_{2}>0$ and let $C\colonequals\{X_{k,n}\geq c\E X_{k,n}\}$.  Assume that $\mathcal{G}_{n}$ is a permutation invariant family of random graphs such that:
\begin{itemize}
\item[(i)] $\P_{\mathcal{P}_{k}(n)}(A\cap C)\leq \epsilon_{1}$
\item[(ii)] $\P_{\mathcal{U}_{k}(n)}(C)\geq 1-\epsilon_{2}$, and
\end{itemize}
Then, for any $\gamma>0$,
\begin{flalign*}
&\P
\Big(\forall\,0\leq k\leq n,\,\,\frac{(1-s)(\E_{\mathcal{P}_{k}(n)}N_{\sigma}-\gamma)}{k+1}
\leq\frac{X_{k+1,n}}{X_{k,n}}
\leq \frac{(1+s')\E_{\mathcal{P}_{k}(n)}N_{\sigma}}{k+1}+n\frac{\gamma}{k+1}\Big)\\
&\qquad\qquad\qquad\qquad\qquad\qquad\qquad\qquad\qquad\qquad\qquad\qquad\qquad\qquad
\geq 1-\frac{n}{\gamma}\Big(\frac{\epsilon_{1}}{c}+\epsilon_{2}\Big).
\end{flalign*}
\end{lemma}
\begin{remark}
If $X_{k}=0$ then $X_{k+1}=0$, so division by zero does not occur in the above event.  If $X_{k+1}=X_{k}=0$, we interpret the event $\{a\leq X_{k+1}/X_{k}\leq b\}$ to always be true, so that it is the whole sample space.
\end{remark}
\begin{proof}
From Lemma \ref{lemma40a},
\begin{equation}\label{etadef}
\P_{\mathcal{U}_{k}(n)}(A)
\leq\frac{1}{c}\P_{\mathcal{P}_{k}(n)}(A\cap C)+\P_{\mathcal{U}_{k}(n)}(C^{c})
\stackrel{(i)\wedge(ii)}{=}\frac{\epsilon_{1}}{c}+\epsilon_{2}
\equalscolon\eta.
\end{equation}
Using Definition \ref{ukdef}, and as usual letting $\E$ denote expected value with respect to the random graph itself,
\begin{equation}\label{pukaeq}
\P_{\mathcal{U}_{k}(n)}(A)
=\sum_{(G,\sigma)\in\Lambda_{k}(n)}\frac{\P(G)1_{(G,\sigma)\in A}}{X_{k,n}(G)}
=\E\sum_{\sigma\colon (G,\sigma)\in\Lambda_{k}(n)}\frac{1_{(G,\sigma)\in A}}{X_{k,n}(G)}.
\end{equation}
So, as usual denoting $\P$ as the probability law with respect to the random graph itself, by Markov's inequality,
\begin{equation}\label{markovapp}
\P\Big(\sum_{\sigma\colon (G,\sigma)\in\Lambda_{k}(n)}\frac{1_{(G,\sigma)\in A}}{X_{k,n}(G)}>\gamma\Big)\stackrel{\eqref{pukaeq}\wedge\eqref{etadef}}{\leq}\eta/\gamma,\qquad\forall\,\gamma>0.
\end{equation}

Consider now the left sum in Lemma \ref{countinglemma}.  We write
$$
\sum_{\abs{S}=k}1_{\{S\,\,\mathrm{is}\,\,\mathrm{independent}\}}\cdot N_{S}\\
=\sum_{\abs{S}=k}[1_{(G,S)\in A^{c}}+1_{(G,S)\in A}]1_{\{S\,\,\mathrm{is}\,\,\mathrm{independent}\}} N_{S}.
$$
From the Definition of $A$, we therefore have
\begin{flalign*}
&(1-s)\E_{\mathcal{P}_{k}(n)}N_{\sigma}\sum_{\abs{S}=k}1_{(G,S)\in A^{c}}1_{\{S\,\,\mathrm{is}\,\,\mathrm{independent}\}}
+\sum_{\abs{S}=k}1_{(G,S)\in A}1_{\{S\,\,\mathrm{is}\,\,\mathrm{independent}\}} N_{S}\\
&\qquad\leq\sum_{\abs{S}=k}1_{\{S\,\,\mathrm{is}\,\,\mathrm{independent}\}} N_{S}
\leq (1+s')X_{k,n}\E_{\mathcal{P}_{k}(n)}N_{\sigma}+\sum_{\abs{S}=k}1_{(G,S)\in A}1_{\{S\,\,\mathrm{is}\,\,\mathrm{independent}\}}N_{S}.
\end{flalign*}
Adding and subtracting $1_{(G,S)\in A}$ in the first term gives
\begin{flalign*}
&(1-s)X_{k,n}\E_{\mathcal{P}_{k}(n)}N_{\sigma}   -(1-s)\E_{\mathcal{P}_{k}(n)}N_{\sigma}  \sum_{\abs{S}=k}1_{(G,S)\in A}1_{\{S\,\,\mathrm{is}\,\,\mathrm{independent}\}}\\
&\qquad\quad
+\sum_{\abs{S}=k}1_{(G,S)\in A}1_{\{S\,\,\mathrm{is}\,\,\mathrm{independent}\}} N_{S}
\leq\sum_{\abs{S}=k}1_{\{S\,\,\mathrm{is}\,\,\mathrm{independent}\}} N_{S}\\
&\qquad\leq (1+s')X_{k,n}\E_{\mathcal{P}_{k}(n)}N_{\sigma}+\sum_{\abs{S}=k}1_{(G,S)\in A}1_{\{S\,\,\mathrm{is}\,\,\mathrm{independent}\}}N_{S}.
\end{flalign*}
The event $(G,S)\in A$ implies that $S$ is an independent set, so we can remove that redundancy.  Also, getting rid of the last part of the first term then gives
\begin{flalign*}
&(1-s)X_{k,n}\E_{\mathcal{P}_{k}(n)}N_{\sigma}   -(1-s)\E_{\mathcal{P}_{k}(n)}N_{\sigma}  \sum_{\abs{S}=k}1_{(G,S)\in A}\\
&\qquad\leq\sum_{\abs{S}=k}1_{\{S\,\,\mathrm{is}\,\,\mathrm{independent}\}} N_{S}
\leq (1+s')X_{k,n}\E_{\mathcal{P}_{k}(n)}N_{\sigma}+\sum_{\abs{S}=k}1_{(G,S)\in A}N_{S}.
\end{flalign*}
Then, dividing by $(k+1)X_{k,n}$, using Lemma \ref{countinglemma}, and using $N_{S}\leq n$ in the last inequality
\begin{flalign*}
&\frac{(1-s)\E_{\mathcal{P}_{k}(n)}N_{\sigma}}{k+1}  -(1-s)\E_{\mathcal{P}_{k}(n)}N_{\sigma}  \frac{1}{k+1}\sum_{\abs{S}=k}\frac{1_{(G,S)\in A}}{X_{k,n}(G)}\\
&\qquad\leq \frac{X_{k+1,n}}{X_{k,n}}
=\frac{1}{(k+1)X_{k,n}}\sum_{\abs{S}=k}1_{\{S\,\,\mathrm{is}\,\,\mathrm{independent}\}} N_{S}\\
&\qquad\leq \frac{(1+s')\E_{\mathcal{P}_{k}(n)}N_{\sigma}}{k+1}+n\frac{1}{k+1}\sum_{\abs{S}=k}\frac{1_{(G,S)\in A}}{X_{k,n}}.
\end{flalign*}
From Markov's inequality \eqref{markovapp}, for any $\gamma>0$, we therefore have
\begin{equation}\label{upperbdp}
\P\Big(\frac{(1-s)(\E_{\mathcal{P}_{k}(n)}N_{\sigma}-\gamma)}{k+1}
\leq\frac{X_{k+1,n}}{X_{k,n}}
\leq \frac{(1+s')\E_{\mathcal{P}_{k}(n)}N_{\sigma}}{k+1}+n\frac{\gamma}{k+1}\Big)\geq 1-\eta/\gamma.
\end{equation}

Finally, taking the union bound over all $0\leq k\leq n$,
$$
\P
\Big(\forall\,0\leq k\leq n,\,\,\frac{(1-s)(\E_{\mathcal{P}_{k}(n)}N_{\sigma}-\gamma)}{k+1}
\leq\frac{X_{k+1,n}}{X_{k,n}}
\leq \frac{(1+s')\E_{\mathcal{P}_{k}(n)}N_{\sigma}}{k+1}+n\frac{\gamma}{k+1}\Big)
\geq 1-\frac{n\eta}{\gamma}.
$$
The proof is concluded by substituting the definition of $\eta$ in \eqref{etadef}.
\end{proof}
\section{Unimodality for Sparse Erd\"{o}s-Renyi Graphs}

\begin{lemma}\label{nchernoff}
Let $(G,\sigma)\in\Lambda_{k}(n)$ be a random sample from $\mathcal{P}_{k}(n)$.  Let $N_{\sigma}$ be the number of vertices in $G\setminus\sigma$ that are not connected to $\sigma$.  Then
$$\P_{\mathcal{P}_{k}(n)}(\abs{N_{\sigma}-\E_{\mathcal{P}_{k}(n)} N_{\sigma}}\geq t\E_{\mathcal{P}_{k}(n)} N_{\sigma})\leq 2e^{- n\min(t,t^{2})(1-\alpha)e^{-d\alpha}/3}\quad\forall\,t\geq0.$$
\end{lemma}
\begin{proof}
Let $\alpha\colonequals k/n$.  Let $v\in\{1,\ldots,n\}$ be a vertex with $v\notin\sigma$.  Let $C_{v}$ be the event that $v$ is connected to some vertex in $\sigma$.  The events $\{C_{v}\}_{v\in\{1,\ldots,n\}\setminus\sigma}$ are independent, and $\P(C_{v}^{c})=(1-p)^{k}$.  So, $N_{\sigma}$ is a binomial random variable with expected value
$$\E_{\mathcal{P}_{k}(n)} N_{\sigma}=(n-k)(1-p)^{k}=(n-k)(1-d/n)^{\alpha n}=(1+o(1))n(1-\alpha)e^{-d\alpha}.$$
So, from Chernoff bounds,
$$\P_{\mathcal{P}_{k}(n)}(\abs{N_{\sigma}-\E_{\mathcal{P}_{k}(n)} N_{\sigma}}\geq t\E_{\mathcal{P}_{k}(n)} N_{\sigma})\leq 2e^{- n\min(t,t^{2})(1-\alpha)e^{-d\alpha}/3},\qquad\forall\,t\geq0.$$
\end{proof}

\begin{proof}[Proof of Theorem \ref{thm1}]
Denote $\alpha\colonequals k/n$.  From Corollary \ref{betterERlowerbound}, with probability at least $1-8\exp\Big(-\frac{n}{2d^{2}(\log d)^{2}}\Big)\equalscolon1-\epsilon_{2}$,
$$X_{k,n}\geq e^{-20nd^{-3/2}(\log d)^{-3/2}}\E X_{k,n}\equalscolon c_{1}\E X_{k,n}.$$
Let $C$ denote this event.

Then, let $A$ be the event that $(G,\sigma)$ satisfies $\abs{N_{\sigma}-\E_{\mathcal{P}_{k}(n)}  N_{\sigma}}>s\E_{\mathcal{P}_{k}(n)}  N_{\sigma}$.  From Lemma \ref{nchernoff},
$$\P_{\mathcal{P}_{k}(n)}(A)\leq 2e^{- n\min(s,s^{2})(1-\alpha)e^{-d\alpha}/3}\equalscolon\epsilon_{1}\quad\forall\,t\geq0.$$

We therefore apply Lemma \ref{ratiolemma} so that, for $\gamma\colonequals e^{-n\epsilon'}$,
\begin{equation}\label{maineq}
\begin{aligned}
&\P
\Big(\forall\,0\leq k\leq n,\,\,\frac{(1-s)(\E_{\mathcal{P}_{k}(n)}N_{\sigma}-\gamma)}{k+1}\leq \frac{X_{k+1,n}}{X_{k,n}}\leq\frac{(1-\gamma+s)\E_{\mathcal{P}_{k}(n)}N_{\sigma}}{k+1}+ \gamma\frac{n}{k+1}\Big)\\
&\quad\geq 1-\frac{n}{\gamma}\Big(\frac{\epsilon_{1}}{c_{1}}+\epsilon_{2}\Big)\\
&\quad\geq
1-ne^{n\epsilon'}2e^{- n\min(s,s^{2})(1-\alpha)e^{-d\alpha}/3}
e^{20nd^{-3/2}(\log d)^{-3/2}}
-ne^{n\epsilon'}8\exp^{-nd^{-2}(\log d)^{-2}}.
\end{aligned}
\end{equation}

So, if we choose $s\colonequals\sqrt{100\frac{20nd^{-3/2}(\log d)^{-3/2}}{(1-\alpha)e^{-d\alpha}}}$, then we can also take $\gamma\colonequals e^{-n\epsilon'}$ for some $\epsilon'>0$.
%
%
%
%
%
%
%
%
%
%
For $d$ sufficiently large, by Theorem \ref{friezethm}, the largest independent set has expected size $n\beta$, where
\begin{equation}\label{betadef}
\beta\colonequals (2/d)(\log d-\log\log d-\log 2+1+o_{d}(1)).
\end{equation}

Let $\delta\colonequals\alpha/\beta$ so that $0<\delta<1$.  If $s<1$, then
\begin{flalign*}
s&\colonequals\sqrt{100\frac{20nd^{-3/2}(\log d)^{-3/2}}{(1-\alpha)e^{-d\alpha}}}
\stackrel{\eqref{betadef}}{\approx}\sqrt{100\frac{20nd^{-3/2}(\log d)^{-3/2}}{(1-\alpha)e^{-2\delta (\log d-\log\log d-\log 2+1)}}}\\
&\approx\sqrt{100\frac{20nd^{-3/2}(\log d)^{-3/2}}{(1-\alpha)d^{-2\delta}(\log d)^{2\delta}}}
\approx  \sqrt{d^{2\delta-3/2}(\log d)^{-3/2-2\delta}}
= d^{\delta-3/4}(\log d)^{-\delta-3/4}.
\end{flalign*}
Similarly, if $s>1$, then
$$s\approx d^{2\delta-3/2}(\log d)^{-2\delta-3/2}.$$

So, if $\delta<3/4$, then $s\approx d^{\delta-3/4}$, and $s\to0$ as $d\to\infty$.  That is, \eqref{maineq} completes the proof when $0<\delta<3/4$ since by Lemma \ref{nchernoff}
$$\E_{\mathcal{P}_{k}(n)}N_{\sigma}/k\approx \frac{1-\alpha}{\alpha}e^{-d\alpha}\approx\frac{1-\alpha}\alpha^{-1}d^{-2\delta}\stackrel{\eqref{betadef}}{\approx} d^{1-2\delta}.$$
As $d\to\infty$, this quantity is larger than $1$ when $\beta<1/2$ and less than $1$ when $\beta>1/2$.  That is, define $\epsilon\colonequals\abs{1-2\delta}$.

In the remaining case that $3/4\leq\delta<1$, we have $s\to\infty$ as $d\to\infty$, and by Lemma \ref{nchernoff},
$$
\frac{s\E_{\mathcal{P}_{k}(n)} N}{k}
\approx \frac{20d^{-3/2}(\log d)^{-3/2}}{\alpha}
\approx \frac{20d^{-3/2}(\log d)^{-3/2}}{d^{-1}\log d}
\approx d^{-1/2}(\log d)^{-5/2}.
$$
That is, $\frac{s\E N}{k}\to0$ as $d\to\infty$.  So, $X_{k+1,n}/X_{k,n}<1/2$ as $n\to\infty$, so \eqref{maineq} completes the proof in the remaining case $\delta>3/4$.

\end{proof}

\begin{proof}[Proof of Theorem \ref{thm2}]
For all $t>0$, with probability at least $1-e^{-t^{2}/(2n)}$, by Lemma \ref{lowerbound},
$$X_{k,n}\geq e^{-t\log(1-\alpha/\beta)}\sqrt{\frac{\beta(1-\alpha)}{\beta-\alpha}}e^{d\alpha}e^{-n[(\beta-\alpha)\log(\beta-\alpha)-\beta\log\beta-(1-\alpha)\log(1-\alpha)-d\alpha^{2}/2]}\E X_{k,n}.$$
Let $C$ denote this event.  Here $\beta n$ is the expected size of the largest independent set in the random graph.

Then, let $A$ be the event that $(G,\sigma)$ satisfies $\abs{N_{\sigma}-\E_{\mathcal{P}_{k}(n)} N_{\sigma}}>s\E_{\mathcal{P}_{k}(n)}  N_{\sigma}$.  From Lemma \ref{nchernoff},
$$\P_{\mathcal{P}_{k}(n)}(A)\leq 2e^{- n\min(s,s^{2})(1-\alpha)e^{-d\alpha}/3}\equalscolon\epsilon_{1}\quad\forall\,t\geq0.$$

We therefore apply Lemma \ref{ratiolemma} so that, for $\gamma\colonequals e^{-n\epsilon'}$,
\begin{equation}\label{maineqb}
\begin{aligned}
&\P
\Big(\forall\,0\leq k\leq n,\,\,\frac{(1-\gamma-s)\E_{\mathcal{P}_{k}(n)}N_{\sigma}}{k+1}\leq \frac{X_{k+1}}{X_{k}}\leq\frac{(1-\gamma+s)\E_{\mathcal{P}_{k}(n)}N_{\sigma}}{k+1}+ \gamma\frac{(n-k)}{k+1}\Big)\\
&\qquad\geq 1-\frac{n}{\gamma}\Big(\frac{\epsilon_{1}}{c_{1}}+\epsilon_{2}\Big)
\geq
1-ne^{n\epsilon'}2e^{- n\min(s,s^{2})(1-\alpha)e^{-d\alpha}/3}
e^{t\log(1-\alpha/\beta)}\\
&\qquad\qquad\quad\cdot\sqrt{\frac{\beta-\alpha}{\beta(1-\alpha)}}e^{d\alpha}e^{n[(\beta-\alpha)\log(\beta-\alpha)-\beta\log\beta-(1-\alpha)\log(1-\alpha)-d\alpha^{2}/2]}
-ne^{n\epsilon'}e^{-t^{2}/(2n)}.
\end{aligned}
\end{equation}

If $s<1$ and we choose $s\colonequals\sqrt{3.03\frac{-\beta\log\beta-(1-\alpha)\log(1-\alpha)+(\beta-\alpha)\log(\beta-\alpha)-d\alpha^{2}/2}{(1-\alpha)e^{-d\alpha}}}$, $t\colonequals\sqrt{n\log n}$, then we can also take $\gamma\colonequals e^{-n\epsilon'}$ for some $\epsilon'>0$.  (If $s>1$, then we remove the square root in the definition of $s$.)  For example
\begin{itemize}
\item When $d=1$, $\beta\approx.728$,  the left side (i.e. the term $\frac{(1-s)\E_{\mathcal{P}_{k}(n)}N_{\sigma}}{k+1}$) exceeds $1$ when $\alpha<.25$, and the right side (i.e. the term $\frac{(1+s)\E_{\mathcal{P}_{k}(n)}N_{\sigma}}{k+1}$) is less than $1$ when $\alpha>.46$.
\item When $d=2$, $\beta\approx.607$,  the left side exceeds $1$ when $\alpha<.194$, and the right side is less than $1$ when $\alpha>.39$.
\item When $d=e$, $\beta\approx.552$,  the left side exceeds $1$ when $\alpha<.172$, and the right side is less than $1$ when $\alpha>.35$.
\end{itemize}
\end{proof}

%
%

%
%
%

%
%
%

\section{Random Regular Graphs}

\begin{lemma}\label{nchernoffreg}
Let $(G,\sigma)\in\Lambda_{k}(n)$ be a random sample from $\mathcal{P}_{k}(n)$.  Let $N_{\sigma}$ be the number of vertices in $G\setminus\sigma$ that are not connected to $\sigma$.  Then
$$\P_{\mathcal{P}_{k}(n)}(\abs{N_{\sigma}-\E_{\mathcal{P}_{k}(n)} N_{\sigma}}\geq t\E_{\mathcal{P}_{k}(n)} N_{\sigma})\leq 2e^{- n\min(t,t^{2})(1-\alpha)e^{-d\alpha}/3}\quad\forall\,t\geq0.$$

\end{lemma}
\begin{proof}
Let $\alpha\colonequals k/n$.  Let $h(x)\colonequals x\log x$.  Let $v\in\{1,\ldots,n\}$ be a vertex with $v\notin\sigma$.  Let $D_{v}$ be the number of vertices in $\sigma$ that $v$ is connected to.  Then $\{D_{v}\}_{v\notin\sigma}$ have the same distribution as a set of independent binomial random variables $\{B_{v}\}_{v\notin\sigma}$ each with parameters $d$ and $p\colonequals\alpha/(1-\alpha)$, conditioned on $\sum_{v\notin\sigma}B_{v}=nd\alpha$.  Then $N_{\sigma}=\sum_{v\notin\sigma}1_{\{B_{v}=0\}}$ and for any $0\leq j\leq nd(1-\alpha)$, (writing $j=\beta n$)
\begin{flalign*}
\P_{\mathcal{P}_{k}(n)}(N_{\sigma}=j)
&=\P(N_{\sigma}=\beta n)
=\frac{\P(\mathrm{exactly}\,\,\beta n\,\,\mathrm{of}\,\,B_{v}\,\,\mathrm{are}\,\,0\,\,\mathrm{and}\,\,\sum_{v\notin\sigma}B_{v}=nd\alpha)}{\P(\sum_{v\notin\sigma}B_{v}=nd\alpha)}\\
&=\frac{\binom{n(1-\alpha)}{\beta n}(1-p)^{nd\beta}\binom{nd(1-\alpha-\beta)}{nd\alpha}p^{nd\alpha}(1-p)^{nd(1-2\alpha-\beta)}}
{\binom{nd(1-\alpha)}{nd\alpha}p^{nd\alpha}(1-p)^{nd(1-2\alpha)}}\\
&=\frac{\frac{[n(1-\alpha)]!}{[\beta n]![n(1-\alpha-\beta)]!}(1-p)^{nd\beta}\frac{[nd(1-\alpha-\beta)]!}{[nd(1-2\alpha-\beta)]![nd\alpha]!}(1-p)^{-d\beta n}}{\frac{[nd(1-\alpha)]!}{[nd(1-2\alpha)]![nd\alpha]!}}\\
\end{flalign*}
\begin{flalign*}
&=\frac{[n(1-\alpha)]!}{[\beta n]![n(1-\alpha-\beta)]!}\frac{[nd(1-\alpha-\beta)]!}{[nd(1-\alpha)]!}\frac{[nd(1-2\alpha)]!}{[nd(1-2\alpha-\beta)]!}\\
&=\exp\Big(n\Big[h(1-\alpha)-h(\beta)-h(1-\alpha-\beta)+h(d(1-\alpha-\beta))\\
&\qquad\qquad-h(d(1-\alpha))+h(d(1-2\alpha))-h(d(1-2\alpha-\beta))\Big]\Big)\\
&=\exp\Big(n\Big[h(1-\alpha)-h(\beta)-h(1-\alpha-\beta)+h(d(1-\alpha-\beta))\\
&\qquad\qquad-h(d(1-\alpha))+h(d(1-2\alpha))-h(d(1-2\alpha-\beta))\Big]\Big).
\end{flalign*} 
\begin{flalign*}
\E_{\mathcal{P}_{k}(n)} N_{\sigma}&=\abs{\sigma^{c}}\P(B_{1}=0)
=(1-\alpha)n\frac{\P(B_{v}=0\,\,\mathrm{and}\,\,\sum_{v\notin\sigma,v\neq 1}B_{v}=nd\alpha)}{\P(\sum_{v\notin\sigma}B_{v}=nd\alpha)}\\
&=(1-\alpha)n\frac{(1-p)^{d}\binom{nd(1-\alpha)-d}{nd\alpha}p^{nd\alpha}(1-p)^{nd(1-2\alpha)-d}}{\binom{nd(1-\alpha)}{nd\alpha}p^{nd\alpha}(1-p)^{nd(1-2\alpha)}}\\
&=(1-\alpha)n\frac{[d[n(1-\alpha)-1]]!}{[dn(1-\alpha)]!}\frac{[nd(1-2\alpha)]!}{[nd(1-2\alpha)-d]!}\\
&=(1-\alpha)n\frac{(nd(1-2\alpha))(nd(1-2\alpha)-1)\cdots(nd(1-2\alpha)-d+1)}{(nd(1-\alpha))(nd(1-\alpha)-1)\cdots(nd(1-\alpha)-d+1)}.
\end{flalign*}
As $n\to\infty$, $\E_{\mathcal{P}_{k}(n)} N_{\sigma}/n\approx [(1-2\alpha)/(1-\alpha)]^{d}=(1-\alpha)(1-p)^{d}\approx (1-\alpha)e^{-\alpha d}$.

\begin{flalign*}
\P_{\mathcal{P}_{k}(n)}(N_{\sigma}\geq j)
&=\sum_{\gamma=\beta,\beta+1/n,\ldots,1}\exp\Big(n\Big[h(1-\alpha)-h(\gamma)-h(1-\alpha-\gamma)+h(d(1-\alpha-\gamma))\\
&\qquad\qquad-h(d(1-\alpha))+h(d(1-2\alpha))-h(d(1-2\alpha-\gamma))\Big]\Big)\\
&\approx \int_{\beta}^{1}\exp\Big(n\Big[h(1-\alpha)-h(\gamma)-h(1-\alpha-\gamma)+h(d(1-\alpha-\gamma))\\
&\qquad\qquad-h(d(1-\alpha))+h(d(1-2\alpha))-h(d(1-2\alpha-\gamma))\Big]\Big)d\gamma.
\end{flalign*}%
As a function of $\gamma$, the integrand is uniformly log-concave, so the result follows.

\end{proof}

\begin{lemma}[\embolden{Lower Bound}]\label{lowerboundreg}
Let $\alpha\colonequals k/n$ and let $\beta n$ be the size of the largest size independent set in the graph.
Then, with probability at least $1-e^{-t^{2}/(2dn)}$,
$$X_{k,n}\geq e^{-t\log(1-\alpha/\beta)}e^{-n[(\beta-\alpha)\log(\beta-\alpha)-\beta\log\beta+(d-1)(1-\alpha)\log(1-\alpha)-(d/2)(1-2\alpha)\log(1-2\alpha)]}\E X_{k,n}.$$
\end{lemma}
\begin{proof}
By Stirling's formula, \cite[Lemma 2.1]{ding16}
$$
\E X_{k,n}= n^{O(1)}\exp\Big(n\Big((d-1)(1-\alpha)\log(1-\alpha)-\alpha\log\alpha-(d/2)(1-2\alpha)\log(1-2\alpha)\Big)\Big).
$$
Also, the size of the largest independent set concentrates around it's expected value by the Azuma-Hoeffding inequality (using the edge-revealing filtration).  So a lower bound for $X_{k,n}$ is given by taking size $k=\alpha n$ subsets of this large independent set, i.e.
$$\P\left(X_{k,n}\geq \binom{\beta n-t}{\alpha n}\right)\geq e^{-t^{2}/[2dn]},\qquad\forall\,t\geq0.$$
Therefore
\begin{flalign*}
\frac{\E X_{k,n}}{\binom{\beta n}{\alpha n}}
&=\frac{(\alpha n)![(\beta-\alpha)n]!}{(\beta n)!}n^{O(1)}\exp\Big(n\Big((d-1)(1-\alpha)\log(1-\alpha)-\alpha\log\alpha\\
&\qquad\qquad-(d/2)(1-2\alpha)\log(1-2\alpha)\Big)\Big).\\
&=n^{O(1)}\exp\Big(n\Big((\beta-\alpha)\log(\beta-\alpha)-\beta\log\beta+(d-1)(1-\alpha)\log(1-\alpha)\\
&\qquad\qquad\qquad\qquad\qquad\qquad-(d/2)(1-2\alpha)\log(1-2\alpha)\Big)\Big).
\end{flalign*}
More specifically,
\begin{flalign*}
\frac{\binom{n}{\alpha n}}{\binom{\beta n-t}{\alpha n}}(1-d/n)^{\binom{\alpha n}{2}}
&=n^{O(1)}\exp\Big(n\Big((\beta-\alpha)\log(\beta-\alpha)-\beta\log\beta+(d-1)(1-\alpha)\log(1-\alpha)\\
&\qquad\qquad\qquad\qquad\qquad\qquad-(d/2)(1-2\alpha)\log(1-2\alpha)\Big)\Big)e^{t\log(1-\alpha/\beta)}.
\end{flalign*}
That is, with probability at least $1-e^{-t^{2}/(2n)}$,
\begin{flalign*}
\frac{\E X_{k,n}}{X_{k,n}}
&\leq \frac{\E X_{k,n}}{\#\,\mbox{size $k$ subsets of the largest independent set}}\\
&\leq n^{O(1)}\exp\Big(n\Big((\beta-\alpha)\log(\beta-\alpha)-\beta\log\beta+(d-1)(1-\alpha)\log(1-\alpha)\\
&\qquad\qquad\qquad\qquad\qquad\qquad-(d/2)(1-2\alpha)\log(1-2\alpha)\Big)\Big)e^{t\log(1-\alpha/\beta)}.
\end{flalign*}

\end{proof}

\begin{proof}[Proof of Theorem \ref{thm5}]
Denote $\alpha\colonequals k/n$.  From Lemma \ref{lowerboundreg},with probability at least $1-e^{-t^{2}/(2dn)}\equalscolon1-\epsilon_{2}$,
$$X_{k,n}\geq e^{-t\log(1-\alpha/\beta)}e^{-n[(\beta-\alpha)\log(\beta-\alpha)-\beta\log\beta+(d-1)(1-\alpha)\log(1-\alpha)-(d/2)(1-2\alpha)\log(1-2\alpha)]}\E X_{k,n}.$$
Let $C$ denote this event.

Then, let $A$ be the event that $(G,\sigma)$ satisfies $\abs{N_{\sigma}-\E_{\mathcal{P}_{k}(n)}  N_{\sigma}}>s\E N_{\sigma}$.  From Lemma \ref{nchernoffreg},
$$\P_{\mathcal{P}_{k}(n)}(A)\leq 2e^{- n\min(s,s^{2})(1-\alpha)e^{-d\alpha}/3}\equalscolon\epsilon_{1}\quad\forall\,t\geq0.$$

We therefore apply Lemma \ref{ratiolemma} so that, for $\gamma\colonequals e^{-n\epsilon'}$, and $t\colonequals\sqrt{n\log n}$
\begin{flalign*}
&\P
\Big(\forall\,0\leq k\leq n,\,\,\frac{(1-s)(\E_{\mathcal{P}_{k}(n)}N_{\sigma}-\gamma)}{k+1}
\leq\frac{X_{k+1,n}}{X_{k,n}}
\leq \frac{(1+s')\E_{\mathcal{P}_{k}(n)}N_{\sigma}}{k+1}+n\frac{\gamma}{k+1}\Big)\\
&\qquad\geq 1-\frac{n}{\gamma}\Big(\frac{\epsilon_{1}}{c_{1}}+\epsilon_{2}\Big)
\geq
1-ne^{n\epsilon'}2e^{- n\min(s,s^{2})(1-\alpha)e^{-d\alpha}/3}
e^{t\log(1-\alpha/\beta)}\\
&\qquad\qquad\qquad\cdot\sqrt{\frac{\beta-\alpha}{\beta(1-\alpha)}}e^{d\alpha}e^{n[(\beta-\alpha)\log(\beta-\alpha)-\beta\log\beta-(1-\alpha)\log(1-\alpha)-d\alpha^{2}/2]}
-ne^{n\epsilon'}e^{-t^{2}/(2n)}.
\end{flalign*}

So, if we choose $s\colonequals\sqrt{3.03\frac{-\beta\log\beta-(1-\alpha)\log(1-\alpha)+(\beta-\alpha)\log(\beta-\alpha)-d\alpha^{2}/2}{(1-\alpha)e^{-d\alpha}}}$ and $t\colonequals\sqrt{n\log n}$, then we can also take $\gamma\colonequals e^{-n\epsilon'}$ for some $\epsilon'>0$.

For $d$ sufficiently large, by Theorem \ref{friezethm} (whose conclusion holds in this case by e.g. \cite[page 21]{wormald99}), the largest independent set has size $n\beta$, where
\begin{equation}\label{betadef2}
\beta\colonequals (2/d)(\log d-\log\log d-\log 2+1).
\end{equation}

Let $\delta\colonequals\alpha/\beta$ so that $0<\delta<1$.  Then
\begin{flalign*}
s&\colonequals\sqrt{100\frac{-\beta\log\beta-(1-\alpha)\log(1-\alpha)+(\beta-\alpha)\log(\beta-\alpha)-d\alpha^{2}/2-d\alpha/n}{(1-\alpha)e^{-d\alpha}}}\\
&\approx\sqrt{100\frac{-\beta\log\beta+\beta(1-\delta)\log(\beta(1-\delta))-d\beta^{2}\delta^{2}/2-d\beta\delta/n}{(1-\alpha)e^{-d\delta\beta}}}\\
&\stackrel{\eqref{betadef2}}{\approx}\sqrt{100\frac{-\beta\log\beta+\beta(1-\delta)\log\beta+\beta(1-\delta)\log(1-\delta)-d\beta^{2}\delta^{2}/2-d\beta\delta/n}{(1-\alpha)e^{-2\delta (\log d-\log\log d-\log 2+1)}}}\\
&=\sqrt{100\frac{-\delta\beta\log\beta+\beta(1-\delta)\log(1-\delta)-d\beta^{2}\delta^{2}/2-d\beta\delta}{(1-\alpha)e^{-2\delta (\log d-\log\log d-\log 2+1)}}}\\
&\approx\sqrt{100\frac{-\delta(2/d)\log d \log(2/d)+(2/d)\log(d)(1-\delta)\log(1-\delta)-4(1/d)\log^{2}(d)\delta^{2}/2}{d^{-2\delta}}}\\
&\approx\sqrt{100\frac{-\delta(2/d)\log d \log(2/d)-4(1/d)\log^{2}(d)\delta^{2}/2}{d^{-2\delta}}}\\
&\approx d^{\delta-1/2}\sqrt{200}\sqrt{\delta\log d}\sqrt{\log(d/2)-\delta\log(d)}
=d^{\delta-1/2}\sqrt{200}\sqrt{\delta(1-\delta)}\log d.
\end{flalign*}

So, if $\delta<1/2$, $s\to0$ as $d\to\infty$.  So, we define $\epsilon\colonequals\abs{\delta-1/2}$ to conclude.

\end{proof}

\section{Random Trees}

Let $Y$ be a random variable defined on trees.  Suppose we use the Aldous-Broder algorithm \cite{broder89,aldous90} to generate a random tree.  For any integer $0\leq k\leq n-1$, let $\mathcal{F}_{k}$ denote the set of edges revealed by the algorithm.  Each time a new edge is added to the tree, there is another piece of information given, so this naturally yields a (Doob) martingale.  For any integer $0\leq k\leq n-1$, let $Y_{k}\colonequals \E(Y\,|\,\mathcal{F}_{k})$.

\begin{lemma}[\embolden{Azuma-Hoeffding for Aldous-Broder}]\label{azumahab}
Assume that $\abs{Y_{k}-Y_{k-1}}\leq 1$ for all $1\leq k\leq n-1$.  Then
$$\P(|Y-E Y|>t\E Y)\leq 2e^{-\frac{t^{2}(\E Y)^{2}}{2(n-1)}},\qquad\forall\,t>0.$$
\end{lemma}
\begin{proof}
Apply the Azuma-Hoeffding Inequality, Lemma \ref{azumah}.
\end{proof}

Let $\sigma\subset\{1,\ldots,n\}$ with $\abs{\sigma}=k$.  Let $\alpha\colonequals k/n$.  Let $v\in\{1,\ldots,n\}$ be a vertex with $v\notin S$.  Let $C_{v}$ be the event that $v$ is connected to some vertex in $\sigma$. Let $N_{\sigma}\colonequals\sum_{v\notin S}1_{C_{v}^{c}}$ be the number of vertices not connected to $\sigma$.

\begin{lemma}\label{treecheaplem}
$$\E_{\mathcal{P}_{k}(n)} N_{\sigma}=n(1-\alpha)^{2}e^{-\alpha/(1-\alpha)}(1+O(1/n)).$$
$$\P_{\mathcal{P}_{k}(n)} (\abs{N_{\sigma}-\E\mathcal{P}_{k}(n) N_{\sigma}}>t)\leq2e^{-(n-1)\frac{t^{2}(1-\alpha)^{4}e^{-2\alpha/(1-\alpha)}}{2}(1+O(\log n/n))},\qquad\forall\,t>0.$$
\end{lemma}
\begin{proof}
The first equality follows from Lemma \ref{lemma31}, i.e. the matrix-tree theorem, since the probability that vertices $\{1,\ldots,k+1\}$ form an independent set, given that the vertices $\{1,\ldots,k\}$ form an independent set is
$$\frac{(n-k-1)^{k}n^{n-k-2}}
{(n-k)^{k-1}n^{n-k-1}}
=\Big(1-\frac{k}{n}\Big)\Big(1-\frac{1}{n-k}\Big)^{k}
=(1-\alpha)e^{-\alpha/(1-\alpha)}(1+o(1)).
$$
So, $\E_{\mathcal{P}_{k}(n)} N_{\sigma}$ is $n-k$ times this number.  The second inequality follows from the first and Lemma \ref{azumahab}.
\end{proof}

\begin{lemma}[\embolden{Lower Bound for Trees}]\label{treelowerbound}
For a uniformly random labelled tree on $n$ vertices, we have the following deterministic bound (i.e. with probability one):
$$X_{k,n}\geq(1+o(1))e^{n[(2-3\alpha)\log(1-\alpha)-(1-2\alpha)\log(1-2\alpha)]}\E X_{k,n}.$$
\end{lemma}
\begin{proof}
Rather than repeating the proof of Lemma \ref{lowerbound}, we use a deterministic bound for $X_{k,n}$ from \cite{wingard95}, as pointed out to use by David Galvin.

From Lemma \ref{lemma31}, and since there are $n^{n-2}$ labelled trees on $n$ vertices,
\begin{equation}\label{three1t}
\E X_{k,n}=\binom{n}{\alpha n}(1-\alpha)^{\alpha n}.
\end{equation}
Theorem 5.1 from \cite{wingard95} says that
$$X_{k,n}\geq\binom{n-k+1}{k}=\binom{n(1-\alpha)-1}{\alpha n}.$$
So, from Stirling's formula,
\begin{flalign*}
\frac{X_{k,n}}{\E X_{k,n}}
&\stackrel{\eqref{three1t}}{\geq}(1+o(1))e^{n[(1-\alpha)\log(1-\alpha)-\alpha\log\alpha-(1-2\alpha)\log(1-2\alpha)]}\\
&\qquad\qquad\qquad\cdot e^{-n[-\alpha\log\alpha-(1-\alpha)\log(1-\alpha)+\alpha\log(1-\alpha)]}\\
&=(1+o(1))e^{n[(2-3\alpha)\log(1-\alpha)-(1-2\alpha)\log(1-2\alpha)]}.
\end{flalign*}

\end{proof}

\subsection{Conditioning Arguments}

The concentration inequality in Lemma \ref{treecheaplem} can be used with Lemma \ref{ratiolemma} to show that roughly the first $30\%$ of the nonzero independent set sequence of a random tree is increasing.  However, the constants appearing in Lemma \ref{treecheaplem} are not sharp.  They can be improved to the following form, yielding a better unimodality result.

\begin{lemma}[{\cite{arratia20}.}]\label{treecheaplemv2}
Let $T$ be a uniformly random tree on $n$ vertices, conditioned on $\sigma$ being an independent set.  Let $N_{\sigma}$ be the number of vertices in $\sigma^{c}$ not connected to $\sigma$.  Then
$$\P_{\mathcal{P}_{k}(n)} (\abs{N_{\sigma}-\E\mathcal{P}_{k}(n) N_{\sigma}}>s\E\mathcal{P}_{k}(n) N_{\sigma}+1)
\leq e^{-\min(s,s^{2})n(1-\alpha)^{2}e^{-\alpha/(1-\alpha)}/3}.
$$
More generally,
$$\P_{\mathcal{P}_{k}(n)}( N_{\sigma}<(1-s)\E\mathcal{P}_{k}(n) N_{\sigma}-1)\leq
e^{-s^{2}n(1-\alpha)^{2}e^{-\alpha/(1-\alpha)}/2},\qquad\forall\,0<s<1
$$
$$\P_{\mathcal{P}_{k}(n)} (N_{\sigma}>(1+s)\E\mathcal{P}_{k}(n) N_{\sigma}+1)\leq
e^{-s^{2}n(1-\alpha)^{2}e^{-\alpha/(1-\alpha)}/(2+s)},\qquad\forall\,s\geq0.
$$
\end{lemma}%

The proof of Lemma \ref{treecheaplemv2} will appear separately in \cite{arratia20}.

\begin{proof}[Proof of Theorem \ref{mainthm2}]
Denote $\alpha\colonequals k/n$.  From Lemma \ref{treelowerbound}, with probability $1$,
$$X_{k,n}\geq (1+o(1))e^{n[(2-3\alpha)\log(1-\alpha)-(1-2\alpha)\log(1-2\alpha)]}$$
Let $C$ denote this event, and let $\epsilon_{2}\colonequals0$.

Then, let $A$ be the event that $(G,\sigma)$ satisfies $N_{\sigma}-\E_{\mathcal{P}_{k}(n)} N_{\sigma}<-s\E_{\mathcal{P}_{k}(n)}  N_{\sigma}-1=\E_{\mathcal{P}_{k}(n)}  N_{\sigma}(-s-1/\E_{\mathcal{P}_{k}(n)}  N_{\sigma})$.  From Lemma \ref{treecheaplemv2},
$$\P_{\mathcal{P}_{k}(n)}(A)\leq e^{-s^{2}n(1-\alpha)^{2}e^{-\alpha/(1-\alpha)}/2},\qquad\forall\,0<s<1.$$

We therefore apply Lemma \ref{ratiolemma} so that, for $\gamma\colonequals e^{-n\epsilon'}$,
\begin{equation}\label{maineqtree}
\begin{aligned}
&\P
\Big(\forall\,0\leq k\leq n,\,\,-\frac{1}{k+1}\Big(1-\frac{\gamma}{\E_{\mathcal{P}_{k}(n)}N_{\sigma}}\Big)
+\frac{(1-s)(\E_{\mathcal{P}_{k}(n)}N_{\sigma}-\gamma)}{k+1}
\leq\frac{X_{k+1,n}}{X_{k,n}}
\Big)\\
&\qquad
\geq 1-\frac{n}{\gamma}\Big(\frac{\epsilon_{1}}{c_{1}}+\epsilon_{2}\Big)
\geq
1-ne^{n\epsilon'}2e^{-s^{2}n(1-\alpha)^{2}e^{-\alpha/(1-\alpha)}/2}\\
&\qquad\qquad\cdot
e^{-n[-(2-3\alpha)\log(1-\alpha)+(1-2\alpha)\log(1-2\alpha)]}e^{t\log(1-\alpha/\beta)}.
\end{aligned}
\end{equation}

So, if we choose $s\colonequals\sqrt{2\frac{-(2-3\alpha)\log(1-\alpha)+(1-2\alpha)\log(1-2\alpha)}{(1-\alpha)^{2}e^{-\alpha/(1-\alpha)}}}$, then we can also take $\gamma\colonequals e^{-n\epsilon'}$ for some $\epsilon'>0$.  The term $-\frac{1}{k+1}\Big(1-\frac{\gamma}{\E_{\mathcal{P}_{k}(n)}N_{\sigma}}\Big)$ in \eqref{maineqtree} is negligible for fixed $\alpha$ and $n$ large since it is of the form $O((\alpha n)^{-1}(1-O(e^{-n\epsilon'}n^{-1})))$.


The left-most term in the quantity \eqref{maineqtree} gives unimodality when it is greater than $1$, i.e. when $\alpha$ satisfies
\begin{equation}\label{alphaeqb}
\begin{aligned}
&\frac{(1+o(1))(1-\alpha)^{2}e^{-\frac{\alpha}{1-\alpha}}}{\alpha}\\
&\qquad\cdot\Big(1-\sqrt{2\frac{-(2-3\alpha)\log(1-\alpha)+(1-2\alpha)\log(1-2\alpha)}{(1-\alpha)^{2}e^{-\alpha/(1-\alpha)}}}\Big)>1.
\end{aligned}
\end{equation}
This inequality holds when $\alpha<.26543$.  (Also $s<1$ when $\alpha<.4$.)
\end{proof}
\begin{remark}\label{nearlyrk}
We note the following in passing.  The right part of the quantity in \eqref{maineqtree} gives unimodality when it is less than $1$, i.e. when $\alpha$ satisfies
\begin{equation}\label{alphaeq}
\begin{aligned}
&\frac{(1+o(1))(1-\alpha)^{2}e^{-\frac{\alpha}{1-\alpha}}}{\alpha}\\
&\qquad\cdot\Big(1+\sqrt{2\frac{-(2-3\alpha)\log(1-\alpha)+(1-2\alpha)\log(1-2\alpha)}{(1-\alpha)^{2}e^{-\alpha/(1-\alpha)}}}\Big)<1.
\end{aligned}
\end{equation}

This inequality holds when $\alpha>.37824$.  This result nearly recovers the result of Levit-Mandrescu, Theorem \ref{levitthm}, which holds for all trees, for all $\alpha>.378095$.
\end{remark}

\section{Appendix: Matrix Tree Theorem}

\begin{lemma}[{\cite{bedrosian64}}]\label{lemma31}
The number of labelled trees on $n\geq k+1$ vertices where vertices $1,\ldots,k$ form an independent set is
$$\frac{1}{n}(n-k)^{k-1}n^{n-k}=(n-k)^{k-1}n^{n-k-1}.$$
\end{lemma}%
\begin{proof}
Consider the adjacency matrix
\begin{table}[htbp!]
\centering
\begin{tabular}{ccccccccc|c|}
 &   &\multicolumn{3}{c}{$k$} &\multicolumn{4}{c}{$n-k$}\\
  &   &\multicolumn{3}{c}{$\overbrace{\hspace{3cm}}$} & \multicolumn{4}{c}{$\overbrace{\hspace{4cm}}$}\\
 \cline{3-9}
  & &  \multicolumn{1}{|c}{$n-k$} & $\cdots$ & \multicolumn{1}{c|}{$0$} & $-1$ & $-1$ & $\cdots$ &  $-1$ \\
 \multirow{ 2}{*}{\begin{rotate}{90} $k$ \end{rotate}} & \multirow{ 3}{*}{\begin{rotate}{90} $\overbrace{\hspace{1.2cm}}$  \end{rotate}}
 & \multicolumn{1}{|c}{$\vdots$} & $\ddots$ &\multicolumn{1}{c|}{$\vdots$} & $\vdots$ & $\vdots$ & $\cdots$ &  $\vdots$ \\
  & & \multicolumn{1}{|c}{$0$}&  $\cdots$ & \multicolumn{1}{c|}{$n-k$}& $-1$ & $-1$ & $\cdots$ & $-1$ \\
   \cline{3-9}
  & &  \multicolumn{1}{|c}{$-1$} &  $\cdots$ & \multicolumn{1}{c|}{$-1$}& $n-1$ & $-1$  &  $\cdots$ & $-1$  \\
\multirow{ 3}{*}{\begin{rotate}{90} $n-k$ \end{rotate}}  &   &  \multicolumn{1}{|c}{$-1$}&  $\cdots$ & \multicolumn{1}{c|}{$-1$} & $-1$  & $n-1$ &  $\cdots$ &  $-1$ \\
& \multirow{ 4}{*}{\begin{rotate}{90} $\overbrace{\hspace{2.3cm}}$  \end{rotate}}
&    \multicolumn{1}{|c}{$\vdots$} & $\cdots$ &\multicolumn{1}{c|}{$\vdots$} & $\vdots$ &  &  $\ddots$ &  $\vdots$ \\
  & &  \multicolumn{1}{|c}{$-1$} &  $\cdots$ & \multicolumn{1}{c|}{$-1$} & $-1$ & $\cdots$  & $-1$ & $n-1$ \\
\cline{3-9}
\end{tabular}
\label{table0}
\end{table}

A basis of eigenvectors exists including vectors of the form
$$(1,0,\ldots,-1,0,\ldots,0\,|\,0,\ldots,0),\qquad(1,\ldots,1\,|\,0\ldots,0,-k,0,\ldots,0).$$

The eigenvalues of these vectors are $n-k$ and $n$ respectively.  The first eigenvalue appears $k-1$ times and the second appears $n-k$ times.

So, the matrix-tree theorem implies that the number of labelled trees where vertices $1,\ldots,k$ form an independent set is
$$\frac{1}{n}(n-k)^{k-1}n^{n-k}=(n-k)^{k-1}n^{n-k-1}.$$
\end{proof}

\medskip
\noindent\textbf{Acknowledgement}.  Thanks to David Galvin for introducing us to this problem, and for suggesting an improvement to Lemma \ref{treelowerbound}, which then led to an improved constant in Theorem \ref{mainthm2}.  Thanks to Larry Goldstein for several helpful discussions.  Thanks to Richard Arratia for explaining to me the Joyal bijection along with several of its improvements.  Thanks to Tim Austin for emphasizing the importance of \cite{coja15} and informing us of the reference \cite{in12}.  Thanks also to Ken Alexander, Sarah Cannon, Will Perkins, and Tianyi Zheng.

\bibliographystyle{amsalpha}
\bibliography{12162011}

\end{document}